\documentclass[11pt,a4paper]{amsart}
\usepackage{amsmath}
\usepackage{amssymb}
\usepackage{amsthm}
\usepackage{latexsym}
\usepackage{pstricks}
\usepackage{amsbsy}
\usepackage{tikz-cd}
\usepackage{mathtools}
\usepackage{faktor}  

\tikzset{curve/.style={settings={#1},to path={(\tikztostart)
    .. controls ($(\tikztostart)!\pv{pos}!(\tikztotarget)!\pv{height}!270:(\tikztotarget)$)
    and ($(\tikztostart)!1-\pv{pos}!(\tikztotarget)!\pv{height}!270:(\tikztotarget)$)
    .. (\tikztotarget)\tikztonodes}},
    settings/.code={\tikzset{quiver/.cd,#1}
        \def\pv##1{\pgfkeysvalueof{/tikz/quiver/##1}}},
    quiver/.cd,pos/.initial=0.35,height/.initial=0}
\usetikzlibrary{calc}
\tikzset{tail reversed/.code={\pgfsetarrowsstart{tikzcd to}}}
\tikzset{2tail/.code={\pgfsetarrowsstart{Implies[reversed]}}}
\tikzset{2tail reversed/.code={\pgfsetarrowsstart{Implies}}}

\usepackage{array}
\usepackage[all]{xy}
\usepackage{amscd}
\usepackage{mathrsfs}
\usepackage{txfonts}
\usepackage{amsfonts}
\usepackage{graphicx}
\usepackage{listings}
\usepackage{url}
\usepackage{bookmark}
\newtheorem{theorem}{Theorem}[section]
\newtheorem{lemma}[theorem]{Lemma}
\newtheorem{corollary}[theorem]{Corollary}
\newtheorem{proposition}[theorem]{Proposition}
\newtheorem{definition}[theorem]{Definition}
\newtheorem{defprop}[theorem]{Definition-Proposition}
\newtheorem{remark}[theorem]{Remark}
\newtheorem{example}[theorem]{Example}

\usepackage{marginnote}

\begin{document}
\newcommand{\abs}[1]{\left\lvert#1\right\rvert}
\newcommand{\news}[1]{{\color{green} #1}}
\newcommand{\new}[1]{{\color{blue} #1}}
\newcommand{\newp}[1]{{\color{purple} #1}}
\newcommand{\old}[1]{{\color{red} #1}}
\definecolor{darkgreen}{rgb}{0.33, 0.42, 0.18}
\newcommand{\rjm}[1]{{\color{darkgreen} #1}}

\newcommand{\proj}{\mathop{\rm proj}\nolimits}%
\newcommand{\inj}{\mathop{\rm inj}\nolimits}%
\providecommand{\Gen}{\mathop{\rm Gen}\nolimits}%
\providecommand{\Cogen}{\mathop{\rm Cogen}\nolimits}%
\providecommand{\Filt}{\mathop{\rm Filt}\nolimits}%
\providecommand{\cone}{\mathop{\rm cone}\nolimits}%
\providecommand{\Sub}{\mathop{\rm Sub}\nolimits}%
\providecommand{\rad}{\mathop{\rm rad}\nolimits}%
\providecommand{\coker}{\mathop{\rm coker}\nolimits}%
\providecommand{\im}{\mathop{\rm im}\nolimits}%
\providecommand{\btop}{\mathop{\rm top}\nolimits}%

\def\A{\mathcal{A}}
\def\C{\mathcal{C}}
\def\D{\mathcal{D}}
\def\P{\mathcal{P}}
\def\Y{\mathbb{Y}}
\def\Z{\mathcal{Z}}
\def\K{\mathcal{K}}
\def\F{\mathcal{F}}
\def\Eps{\mathcal{E}}
\def\T{\mathcal{T}}
\def\U{\mathcal{U}}
\def\V{\mathcal{V}}
\def\W{\mathcal{W}}
\def\E{\mathsf{E}}
\def\M{\mathsf{M}}
\def\N{\mathsf{N}}
\def\X{\mathsf{X}}
\def\L{\mathsf{L}}
\def\Mcluster{\mathfrak{M}}

\def\PP{{\mathbb P}}
\providecommand{\add}{\mathop{\rm add}\nolimits}%
\providecommand{\End}{\mathop{\rm End}\nolimits}%
\providecommand{\Ext}{\mathop{\rm Ext}\nolimits}%
\providecommand{\Yext}{\mathop{\rm Yext}\nolimits}%
\providecommand{\Hom}{\mathop{\rm Hom}\nolimits}%
\providecommand{\ind}{\mathop{\rm ind}\nolimits}%
\providecommand{\pd}{\mathop{\rm pd}\nolimits}%
\providecommand{\id}{\mathop{\rm id}\nolimits}%
\providecommand{\gldim}{\mathop{\rm gl.dim}\nolimits}%
\newcommand{\modd}{\mathop{\rm mod}\nolimits}%
\newcommand{\res}{\mathop{\mathsf{res}}\nolimits}%
\newcommand{\stautilt}{\textnormal{s$\tau$-tilt }}
\newcommand{\taugenmin}{\textnormal{$\tau$-gen-min }}
\newcommand{\ftors}{\textnormal{f-tors }}

\title{Mutation of $\tau$-exceptional sequences for acyclic quivers over local algebras}
\author[I. Nonis]{Iacopo Nonis}
\address{School of Mathematics \\ 
University of Leeds \\ 
Leeds, LS2 9JT \\ 
United Kingdom \\
}
\email{mmin@leeds.ac.uk}

\begin{abstract}
    Let $k$ be an algebraically closed field. Let $R$ be a local commutative finite dimensional $k$-algebra and let $Q$ be a quiver with no loops or oriented cycles. We show that mutation of $\tau$-exceptional sequences over $\Lambda = R\otimes_k kQ \cong RQ$ in the sense of Buan, Hanson, and Marsh coincides with the classical mutation of exceptional sequences defined by Crawley-Boevey and Ringel. In particular, the braid group acts transitively on the set of complete $\tau$-exceptional sequences in $\modd{\Lambda}$.
\end{abstract}
\maketitle
\tableofcontents  

\section*{Introduction}

An object in an abelian or triangulated category is called \textit{exceptional} if its endomorphism algebra is a division ring and it has no self-extension in any degree. An \textit{exceptional sequence} is a sequence of exceptional objects satisfying some $\Hom$ and $\Ext$ vanishing conditions. Exceptional sequences were first considered in the algebraic geometry context in \cite{Bondal, Gorodentsev, GR} together with the notion of \textit{mutation} on such sequences.  Bondal \cite{Bondal} proved that the braid group acts on the set of exceptional sequences over a triangulated category.  

Later, exceptional sequences were considered in the setting of hereditary finite dimensional algebras by Crawley-Boevey \cite{Boevey-ExcSeq} and Ringel \cite{braid_group_action_on_exc_seq_Ringel}. The \textit{perpendicular categories} studied by Geigle and Lenzing \cite{Geigle_Lenzing_Perp_Subcat} play a central role in this theory. If the length of an exceptional sequence equals the number of simple modules the sequence is said to be \textit{complete}. Complete sequences always exist over hereditary algebras, and the braid group acts transitively on the set of complete exceptional sequences \cite{Boevey-ExcSeq}. 

Motivated by cluster-tilting theory, Adachi, Iyama, and Reiten introduced  \textit{$\tau$-tilting theory} for finite dimensional algebras \cite{tau-tiling-theory}. They defined a $\Lambda$-module $M$ to be $\tau$-\textit{rigid} if $\Hom_\Lambda(M,\tau M) = 0$, where $\tau$ denotes the Auslander-Reiten (AR) translation in $\modd \Lambda$. As a consequence of the AR duality, a $\tau$-rigid module $M$ is rigid (i.e. $\Ext_\Lambda^1(M, M) = 0 $). In general, the converse holds for hereditary algebras. Hence, indecomposable $\tau$-rigid modules, indecomposable rigid modules, and exceptional modules coincide over hereditary algebras. 

Let $M^{^\perp}$ be the full subcategory of $\modd \Lambda$ consisting of all modules $X$ such that $\Hom_\Lambda(M,X) = 0$, and define ${^\perp}M$ dually. For a basic $\tau$-rigid $\Lambda$-module $M$, Jasso \cite{Jasso_Reduction} defined the $\tau$\textit{-perpendicular category} of $M$ to be $J(M):= M^{\perp}\cap {^\perp \tau M }$. In particular, $\tau$-perpendicular categories and Geigle-Lenzing perpendicular categories coincide over hereditary algebras. 

Jasso proved that every $\tau$-perpendicular category $J(M)$ is equivalent to the module category of a finite dimensional algebra with $\abs{\Lambda}-\abs{M}$ 
simple modules, where $\abs{-}$ denotes the number of indecomposable direct summands. Moreover, $\tau$-perpendicular subcategories are functorially finite (see \cite[Prop. 4.2]{Rigid_ICE-subcategories} and \cite[Lemma 4.7]{Enomoto-Sakai}) and wide (see \cite[Thm. 4.12]{Beyond_tau-tilting_theory}) subcategories of $\modd\Lambda$. It follows that a module $M$ which is $\tau$-rigid in a $\tau$-perpendicular subcategory $\W$ is rigid in $\modd\Lambda$. However, in general, such a module is not necessarily $\tau$-rigid in $\modd\Lambda$. 

Following these ideas, Buan and Marsh \cite{tauExcSeq_BM} recursively defined a sequence $(M_1,\cdots,M_t)$ of indecomposable $\Lambda$-modules to be a $\tau$\textit{-exceptional sequence} if $M_t$ is $\tau$-rigid in $\modd\Lambda$ and $(M_1,\cdots, M_{t-1})$ is a $\tau$-exceptional sequence in $J(M_t)$. If $t = \abs{\Lambda}$, the sequence is said to be \textit{complete}. In particular, complete $\tau$-exceptional sequences always exist over arbitrary finite dimensional algebras. We often refer to a projective module in a $\tau$-perpendicular subcategory $\W$ as \textit{relative projective}.

More generally, a \textit{signed $\tau$-exceptional sequence} is a $\tau$-exceptional sequence in the full subcategory $\C(\Lambda):= \modd\Lambda\oplus(\modd\Lambda)[1]$ of the bounded derived category $D^b(\modd\Lambda)$ in which every relative projective object can be ``signed'' (see \cite[Def. 1.3]{tauExcSeq_BM}). Signed $\tau$-exceptional sequences generalize the \textit{signed exceptional sequences} for hereditary algebras defined by Igusa and Todorov in \cite{signed_exc_seq}.  It was proved in \cite{tauExcSeq_BM} that complete signed $\tau$-exceptional sequences are in bijection with ordered support $\tau$-tilting objects. This result generalizes the bijection between complete signed exceptional sequences and ordered cluster-tilting objects in the cluster category \cite{Tilting_Theory_and_Cluster_Combinatorics} of a hereditary algebra established in \cite{signed_exc_seq}. 

Signed exceptional sequences were originally introduced to describe morphisms in the \textit{cluster morphism category}, a category whose objects are finitely generated wide subcategories of the module category of a hereditary algebra. Later, the definition of the cluster morphism category was extended to the $\tau$-tilting finite case in \cite{a_category_of_wide_subcategories} and given the name of $\tau$-cluster morphism category in \cite{Hanson-Igusa_tau-cluster_morphism_categories_and_picture_groups}, where they proved that the classifying space of this category is a cube complex. Moreover, in many cases,  the classifying space of the $\tau$-cluster morphism category is a $K(\pi, 1)$ space with $\pi$ being the picture group of the algebra \cite{igusa2016picture} (see also \cite{Pairwise_compatibility_for_2-simple_minded_collections, barnard2022exceptional}). Finally, the definition of $\tau$-cluster morphism category of an arbitrary finite dimensional algebra was given by Buan and Hanson in \cite{tau-perpendicular_wide_subategories}. Its objects are $\tau$-perpendicular subcategories and morphisms can be described in terms of signed $\tau$-exceptional sequences. The $\tau$-cluster morphism category has also been studied using silting theory \cite{borve21_two-term} and from a geometric perspective \cite{a_geometric_perspective_on_the_tau_cluster_morphism_category, Kaipel_the_category_of_a_partitioned_fan}.

Representations of acyclic quivers $Q$ over the dual numbers were first studied by Ringel and Zhang in \cite{RepOfQuiversOverTheDualNumbers}. This is equivalent to studying the tensor product of $R = k[x]/(x^2)$ with $kQ$ over the field $k$. Subsequently, Geiss, Leclerc, and Schr\"{o}er \cite{Quivers_with_relations_for_symmetrizable_Cartan_matrices_I} studied the algebras $R\otimes_k kQ$, where $R = k[x]/(x^t)$, for $t\geq 2$. 

Recently, the author \cite{nonis2025tau} proposed a further generalization in which he studied (signed) $\tau$-exceptional sequences over the algebra $\Lambda = R\otimes_k kQ$, where $R$ is a finite dimensional local commutative $k$-algebra (we denote the tensor product $\otimes_k$ by $\otimes$). Notice that $\Lambda$ is isomorphic to $RQ$. Every indecomposable $\tau$-rigid $\Lambda$-module is isomorphic to $\Lambda\otimes_{kQ}M$, for some indecomposable $\tau$-rigid $kQ$-module $M$. Moreover, every $\tau$-perpendicular subcategory of $\modd \Lambda$ is equivalent to the module category of $R\otimes kQ'$, for an acyclic quiver $Q'$. Based on these results, the author established a bijection between complete (signed) ($\tau$-)exceptional sequences in $\modd kQ$ and complete (signed) $\tau$-exceptional sequences in $\modd \Lambda$ (see Section \ref{sec1} and Section \ref{sec2}). As a consequence, he proved that the $\tau$-cluster morphism categories of $kQ$ and $\Lambda$ are equivalent. 

Inspired by the triangulated case and the hereditary case, Buan, Hanson, and Marsh \cite{BHM_mutation} introduced a notion of \textit{mutation} for $\tau$-exceptional sequences (BHM mutation) using a notion of mutation of $\tau$-exceptional pairs (i.e. $\tau$-exceptional sequences of length two). In contrast to the triangulated and hereditary cases, there are finite dimensional $k$-algebras in which $\tau$-exceptional pairs can be mutated only in one direction; see \cite[Section 8]{BHM_mutation}. Hence, Buan, Hanson, and Marsh introduced the notion of \textit{left} (resp. \textit{right)} \textit{mutable} $\tau$-exceptional pairs (see \cite[Section 4]{BHM_mutation}) and they proved that every $\tau$-exceptional pair is both right and left mutable when $\Lambda$ is hereditary, $\tau$-tilting finite or has rank two. 

A $\tau$-exceptional sequence $\M = (M_1,\cdots, M_t)$ is \textit{left} (resp. \textit{right}) \textit{i-mutable} (for $1\leq i\leq t-1$) if the $\tau$-exceptional pair $(M_i,M_{i+1})$ is left (resp. right) mutable in the iterated $\tau$-perpendicular subcategory $$J(M_{i+2},\cdots, M_{t}) := J_{J(M_{i+3},\cdots, M_t)}(M_{i+2}).$$ An algebra $\Lambda$ is called \textit{mutation complete}  if every $\tau$-exceptional sequence $(M_1,\cdots, M_t)$ is both left and right $i$-mutable (for $1\leq i\leq t-1$). As a consequence of the above result, if $\Lambda$ is hereditary, $\tau$-tilting finite or has rank two, then it is mutation complete. Over hereditary algebras, exceptional sequences and $\tau$-exceptional sequences coincide, and therefore they are mutable in the sense of Crawley-Boevey and Ringel \cite{Boevey-ExcSeq, braid_group_action_on_exc_seq_Ringel}. Buan, Hanson, and Marsh showed that in the hereditary case BHM mutation of $\tau$-exceptional sequences coincides with the classical mutation of exceptional sequences; see \cite[Thm. 6.1]{BHM_mutation}.

Motivated by the BHM mutation of $\tau$-exceptional sequences and the results established in \cite{nonis2025tau}, this paper aims to study mutation of $\tau$-exceptional sequences over $\Lambda = R\otimes_k kQ \cong RQ$, where $R$ is a finite dimensional local commutative $k$-algebra.

We work over an algebraically closed field $k$. In Section \ref{sec1} and Section \ref{sec2}, we recall some of the main results from \cite{tauExcSeq_BM, nonis2025tau}. In Section \ref{secMutation}, we recall the terminology from \cite{BHM_mutation} and we study mutation of $\tau$-exceptional sequences over $\Lambda = R\otimes kQ$ in the sense of Buan, Hanson, and Marsh. The following is our first main result which gives a new class of algebras that are mutation complete. 

\begin{proposition}[{ Prop. \ref{left-right_mutable_pair_for_Lambda}, Cor. \ref{Lambda_is_mutation_complete}}]
    Let $\Lambda = R\otimes kQ$. Then, every $\tau$-exceptional pair is both left and right mutable. As a consequence, $\Lambda$ is mutation complete. 
\end{proposition}

Denote by $\Lambda\otimes_{kQ}-$ the induction functor from $\modd kQ$ to $\modd \Lambda$. Every $\tau$-exceptional sequence in $\modd \Lambda$ is of the form $\Lambda\otimes_{kQ}\M$ for a unique ($\tau$-)exceptional sequence $\M$ in $\modd kQ$ (see \cite[Thm. 5.10]{nonis2025tau}). Using the argument analogous to that used to prove that the BHM mutation of $\tau$-exceptional sequence and the classical mutation of exceptional sequences coincide over hereditary algebras \cite[Section 6]{BHM_mutation}, we prove the main result of this paper. 

\begin{theorem}[{Theorem \ref{mutation_in_Lambda}}]\label{thm0.2}
        Let $\Lambda = R\otimes kQ$ and let $\Lambda\otimes_{kQ}\M$ be a complete $\tau$-exceptional sequence in $\modd \Lambda$. Then $$\varphi_i(\Lambda\otimes_{kQ}\M) = \Lambda\otimes_{kQ}\sigma_i(\M)$$ for all $1\leq i \leq n-1$, where $\varphi_i$ and $\sigma_i$ denote the BHM mutation and the classical mutation of exceptional sequences, respectively.  
\end{theorem}

Since the braid group acts transitively on the set of complete exceptional sequences over hereditary algebras, the following consequence of Theorem \ref{thm0.2} provides a new class of algebras for which the braid group acts transitively on the set of complete $\tau$-exceptional sequences. 

\begin{corollary}[{Corollary \ref{transitive_action_on_tau-exc_seq_over_Lambda}}]
    There is a transitive braid group action on the set of complete $\tau$-exceptional sequences in $\modd\Lambda$. 
\end{corollary}

In Section \ref{secRQ-lattices} we recall the notion of $RQ$-lattices and $R$-exceptional sequences from \cite{CB_RigidIntegralRep} and we prove that $R$-exceptional sequences and $\tau$-exceptional sequences in $\modd \Lambda$ coincide (see Proposition \ref{tau-exc=R-exc}). As a consequence, the braid action on complete R-exceptional sequences in the sense of \cite{CB_RigidIntegralRep} and the braid action on complete $\tau$-exceptional sequences in $\modd\Lambda$ coincide (see Theorem \ref{braid_action_on_R-exc=braid_action_on_tau-exc}). We conclude with Section \ref{secExample} in which we present a detailed example that explains the theory developed in this work.

\section{$\tau$-rigid modules and wide subcategories}\label{sec1}

Let $\Lambda$ be a basic finite dimensional algebra over an algebraically closed field $k$. Denote by $\modd\Lambda$ the category of finitely generated (left) $\Lambda$-modules and let $\proj\Lambda$ be the full subcategory of projective $\Lambda$-modules. Given a subcategory $\mathcal{X}\subseteq\modd\Lambda$, we denote by $\ind(\mathcal{X})$ the set of isomorphism classes of indecomposable objects in $\mathcal{X}$. We define $\mathcal{X}^{\perp} = \{Y\in\modd\Lambda \mid \Hom_\Lambda(\mathcal{X}, Y) = 0\}$, and define $^{\perp}\mathcal{X}$ dually. $\Gen\mathcal{X}$ (resp. $\Cogen\mathcal{X}$) denotes the smallest subcategory containing $\mathcal{X}$ and being closed under factor objects (resp. subobjects). For $M\in\modd\Lambda$, we denote by $\add M$ the subcategory generated by direct summands of finite direct sums of $M$. We write $\Gen M$ (resp. $\Cogen M$) for $\Gen(\add M)$ (resp. $\Cogen(\add M)$). Any $\Lambda$-module $X$ is considered to be basic when possible and we will denote by $\abs{X}$ the number of indecomposable direct summands of $X$. We let $\abs{\Lambda} = n$. For an arbitrary module category $\W$, let $\tau_\W$ (or simply $\tau$ if no confusion appears) denote the Auslander-Reiten translation in $\W$.

Now let $\Lambda = R\otimes kQ$ unless stated otherwise. It is well-known there is a functor $\Lambda\otimes_{kQ}-: \modd{kQ}\to \modd\Lambda$ known as the \textit{induction functor}. 
The $\Lambda$-modules in the image of the induction functor are called \textit{induced modules}. This section aims to review some of the main results from \cite{nonis2025tau} on $\tau$-rigid $\Lambda$-modules and wide subcategories of $\modd \Lambda$.

We begin by describing $\tau$-rigid $\Lambda$-modules. These are precisely the induced $\Lambda$-modules arising from $\tau$-rigid $kQ$-modules. 

\begin{proposition}[{\cite[Prop. 1.13]{nonis2025tau}}]\label{tau-rigid-Lambda1:1tau-rigid-kQ}
    The induction functor $\Lambda\otimes_{kQ}-$ induces a bijection between the set of isoclasses of indecomposable $\tau$-rigid $kQ$-modules and the set of isoclasses of indecomposable $\tau$-rigid $\Lambda$-modules.
\end{proposition}

Let $\Lambda$ be a finite dimensional algebra. By \cite[Prop. 4.12]{Rigid_ICE-subcategories}, a functorially finite wide subcategory $\W\subseteq\modd \Lambda$ is equivalent to the module category of a finite dimensional algebra. The next result relates the projective dimension of $M$ as a $\Lambda$-module and the projective dimension of $M$ as an object in $\W$. We denote by $\pd_\W$ the projective dimension in $\W$.

\begin{proposition}[{\cite[Prop. 3.7]{nonis2025tau}}]\label{pd_preserved}
   Let $\Lambda$ be a finite dimensional algebra. Let $m\geq 0$ and let $M\in\modd\Lambda$ with $\pd M = m$. Suppose $M$ lies in a functorially finite wide subcategory $\W\subseteq \modd \Lambda$. Then,  $\pd_\mathcal{W} M \leq m$.
\end{proposition}

Let $\Lambda$ be a finite dimensional algebra. Following \cite{tau-tiling-theory}, a pair $(M,P)$ of $\Lambda$-modules is called \textit{$\tau$-rigid} provided $M$ is $\tau$-rigid, $P\in\proj\Lambda$, and $\Hom_\Lambda(P,M) = 0$. Buan and Marsh \cite{tauExcSeq_BM} considered the corresponding object $M\oplus P[1]$ in the full subcategory $\C(\Lambda) := \modd{\Lambda}\oplus \modd{\Lambda}[1]$ of the derived category $D^b(\modd\Lambda)$. These objects will be referred to as \textit{support $\tau$-rigid}. 

Dually, an object $M\oplus I[-1]$ in $\C(\Lambda)[-1]$ is called \textit{support $\tau^{-1}$-rigid} if $M$ is support $\tau^{-1}$-rigid (i.e. $\Hom_\Lambda(\tau^{-1}M,M) = 0$), $I$ is an injective $\Lambda$-module, and $\Hom_\Lambda(M,I) = 0$. 

\begin{definition}
    Let $ U= M\oplus P[1]$ and $V = N\oplus I[-1]$ be support $\tau$-rigid support in $\C(\Lambda)$ and support $\tau^{-1}$-rigid in $\C(\Lambda)[-1]$, respectively. 
    \begin{enumerate}
        \item[(a)] \cite[Def. 3.3]{Jasso_Reduction} The \emph{$\tau$-perpendicular category} of $U$ is $J(U) := M^\perp \cap {^\perp\tau M}\cap P^\perp$; 
        \item[(b)] \cite[Def. 3.4]{tau-perpendicular_wide_subategories} The \emph{$\tau^{-1}$-perpendicular category} of $V$ is $J^d(V) := {^\perp N}\cap (\tau ^{-1} N)^\perp \cap {^\perp N}$. 
    \end{enumerate}
\end{definition}

\begin{remark}
    Let $U$ be a support $\tau$-rigid object in $\C(\Lambda)$. Then, the $\tau$-perpendicular category $J(U)$ coincides with the $\tau$-perpendicular category $J(M)$ for some $\tau$-rigid $\Lambda$-module $M$ (see for example \cite[Lemma 6.5]{nonis2025tau}). 
\end{remark}

Let $U$ be a $\tau$-rigid $\Lambda$-module. We say that an object $N\oplus Q[1]\in \C(J(U))\subseteq \C(\Lambda)$ is \textit{support $\tau$-rigid} in $\C(J(U))$ if the corresponding object in $\C(\Gamma_U)$ is support $\tau$-rigid, that is $N$ is $\tau_{J(U)}$-rigid, $Q$ lies in $\proj J(U)$, and $\Hom_{J(U)}(Q,N) = 0$. Let $f_U$ denote the torsion-free functor relative to the torsion pair $(\Gen U, U^{\perp})$. They established the following bijection.

\begin{theorem}[{\cite[Thm. 2.8]{BHM_mutation}}]\label{BM_bijection}
    Let $\Lambda$ be a finite dimensional algebra and let $U = M\oplus P[1]$ be a support $\tau$-rigid object in $\C(\Lambda)$. Then there is a bijection 
            \[
            \begin{tikzcd}
            {\left\{V\in \ind\C(\Lambda)\mid V\oplus U \text{ support }\tau\text{-rigid} \right\}} 
            \arrow[d, "\mathcal{E}_U"] \\
            {\left\{ W\in \ind\C(J(U)) \mid W \text{ support }\tau_{J(U)}\text{-rigid}\right\}}.
            \end{tikzcd}
            \]
        \begin{enumerate}
            \item[(a)] For $V\in\ind\C(\Lambda)$ with $V\oplus U$ support $\tau$-rigid, we have $\mathcal{E}_U(V)\in (\proj J(U))[1]$ if and only if $V\in \Gen M$ or $V\in(\proj\Lambda)[1]$; 
            \item[(b)] If $U\in \proj\Lambda[1]$ and $V\in \ind\C(\Lambda)$ with $V\oplus U$ support $\tau$-rigid, we have $\mathcal{E}_U(V) = V$. Equivalently, if $W$ is a $\tau_{J(P[1])}$-rigid module, then $\mathcal{E}^{-1}_{P[1]}(W) = W$;  
            \item[(c)] If $V\in (\ind\bmod{\Lambda})\setminus \Gen M$ with $V\oplus U$ support $\tau$-rigid, then $\mathcal{E}_U(V) = \mathcal{E}_M(V) = f_MV$. 
        \end{enumerate}
\end{theorem}

The bijection described in Theorem \ref{BM_bijection} is compatible with the induction functor in the following way. 

\begin{proposition}[{\cite[Prop. 6.14]{nonis2025tau}}]\label{E-version3.7}
    Let $U = M\oplus P[1]\in \C(kQ)$ be support $\tau$-rigid and let $V\in \ind \C(kQ)$ be such that $U\oplus V$ is also support $\tau$-rigid in $\C(kQ)$. Then, $\Lambda\otimes_{kQ}(U\oplus V)$ is support $\tau$-rigid in $\C(\Lambda)$ and there is an isomorphism
    \begin{equation}\label{E-map}
         \Eps_{\Lambda\otimes_{kQ}U}(\Lambda\otimes_{kQ}V)\cong \Lambda\otimes_{kQ}\Eps_U(V)
    \end{equation}
    in $\C(\Lambda)$. In other words, there is a commutative diagram of bijections 
    \[
\begin{tikzcd}[ampersand replacement=\&, cramped, row sep=large, column sep=scriptsize, 
every matrix/.append style={font=\small}]
{\left\{ V \in\ind\C(kQ)\;\middle|\; \begin{tabular}{@{}l@{}} $V\oplus U$ is support \\ $\tau_{kQ}$-rigid \end{tabular} \right\}} \&\& {\left\{W\in \ind \C(J(U))\;\middle|\; \begin{tabular}{@{}l@{}} $W$ is support \\$\tau_{J(U)}$-rigid \end{tabular} \right\}} \\
{\left\{ Y \in \ind\C(\Lambda) \;\middle|\; \begin{tabular}{@{}l@{}} $Y\oplus(\Lambda\otimes_{kQ}U)$ is \\ support $\tau_{\Lambda}$-rigid \end{tabular} \right\}} \&\& {\left\{Z\in \ind \C(J(\Lambda\otimes_{kQ}U))\;\middle|\; \begin{tabular}{@{}l@{}} $Z$ is support \\$\tau_{J(\Lambda\otimes_{kQ}U)}$-rigid \end{tabular} \right\}}.
\arrow["{\mathcal{E}_U}", from=1-1, to=1-3]
\arrow["{\Lambda\otimes_{kQ}-}"', from=1-1, to=2-1]
\arrow["{\Lambda\otimes_{kQ}-}", from=1-3, to=2-3]
\arrow["{\mathcal{E}_{\Lambda\otimes_{kQ}U}}", from=2-1, to=2-3]
\end{tikzcd}
\]
\end{proposition}

Recall that $\mathcal{P}(^\perp\tau M)$ is the full subcategory of $^\perp\tau M$ consisting of $\Ext$-projective modules in $^\perp\tau M$, i.e. the modules $X$ in $^\perp\tau M$ such that $\Ext^1_\Lambda(X,{^\perp\tau}M) = 0$. The \textit{Bongartz completion} of $M$, denoted by $B_M$, is given by the direct sum of all indecomposable $\Ext$-projective modules in $^\perp\tau M$. 

\begin{proposition}\label{indprojJ(M)1-1indprojJ(Lambda_otimes_M)}
    Let $M\in \modd kQ$ be $\tau$-rigid. Then, there is a commutative diagram of bijections
    \[
        \begin{tikzcd}[ampersand replacement=\&, cramped, row sep=large, column sep=scriptsize, 
        every matrix/.append style={font=\small}]
        \mathcal{P}(^\perp\tau M)\setminus \mathrm{ind.}\add(M) 
        \&\& \mathrm{ind.proj}J(M) \\
        \mathcal{P}(^\perp\tau(\Lambda\otimes_{kQ}M))\setminus \mathrm{ind.}\add(\Lambda\otimes_{kQ}M)
        \&\& \mathrm{ind.proj}J(\Lambda\otimes_{kQ}M).
        \arrow["{\mathcal{E}_M}", from=1-1, to=1-3]
        \arrow["{\Lambda\otimes_{kQ}-}"', from=1-1, to=2-1]
        \arrow["{\Lambda\otimes_{kQ}-}", from=1-3, to=2-3]
        \arrow["{\mathcal{E}_{\Lambda\otimes_{kQ}M}}", from=2-1, to=2-3]
        \end{tikzcd}
    \]

\end{proposition}

\begin{proof}
The horizontal arrows are bijections by \cite[Lemma 4.9]{tauExcSeq_BM}. Since $\Eps_M = f_M$ and $\Eps_{\Lambda\otimes_{kQ}M} = f_{\Lambda\otimes_{kQ}M}$ by Theorem \ref{BM_bijection}(c), the result follows from \cite[Prop. 4.6 and Prop. 4.8]{nonis2025tau}.
\end{proof}

A $\tau$-perpendicular category $J(U)$ is equivalent to the module category of a finite dimensional algebra $\Gamma_U$. The next important result provides a precise description of the algebra $\Gamma_U$ when $J(U)$ is a $\tau$-perpendicular subcategory of $\modd{\Lambda}$, where $\Lambda = R\otimes kQ$. 

\begin{theorem}[{\cite[Thm. 4.10]{nonis2025tau}}]\label{J(Lambda_otimes_kQ)=mod(R_otimes_kQ')}
      Let $M$ be a basic $\tau$-rigid $kQ$-module and let $\Lambda\otimes_{kQ}M$ be the corresponding basic $\tau$-rigid $\Lambda$-module. Then, $J(\Lambda\otimes_{kQ}M)\simeq \modd\Gamma_{\Lambda\otimes_{kQ}M}$, with $$ \Gamma_{\Lambda\otimes_{kQ}M} \cong R\otimes \Gamma_M \quad \text{and} \quad \Gamma_M = kQ_M$$
      for an acyclic quiver $Q_M$. 
\end{theorem}

Let $\T\subseteq \modd\Lambda$ be a functorially finite torsion class. Recall from \cite{Noncrossing_partitions_and_representations_of_quivers,Torsion-classes-wide-subcategories-and-localisations} that the \emph{left finite wide subcategory} of $\modd\Lambda$ corresponding to $\T$ is 
$$\W_L(\T):= \{X\in \T \mid \forall (Y\in \T, g: Y\to X), \ker(g)\in \T\}.$$ Dually, given a functorially finite torsion-free class $\F\subseteq \modd\Lambda$, the \emph{right finite wide subcategory} of $\modd\Lambda$ corresponding to $\F$ is 
$$\W_R(\F):= \{X\in \T \mid \forall (Y\in \T, g:  X\to Y), \coker(g)\in \F\}.$$
In particular, every left and right finite wide subcategory of $\modd\Lambda$ is a $\tau$-perpendicular subcategory \cite[Lemma 4.3, Prop. 6.15(i)]{tau-perpendicular_wide_subategories}.

Now, let $\Lambda = R\otimes kQ$. We conclude this section by recalling the relationship between $\tau$-perpendicular subcategories of $\modd kQ$ and $\tau$-perpendicular subcategories of $\modd{\Lambda}$. 

\begin{theorem}[{\cite[Prop. 6.25, Cor. 6.26]{nonis2025tau}}]\label{bijection_of_tau-perp_subcategories}
    Let $\Lambda = R\otimes kQ$. The following statements hold. 
    \begin{enumerate}
        \item[(i)]For a basic $\tau$-rigid $kQ$-module $M$, the assignment $J(M)\mapsto J(\Lambda\otimes_{kQ}M)$ defines a bijection between the $\tau$-perpendicular subcategories of $\modd{kQ}$ and $\modd{\Lambda}$; 
        \item[(ii)] The $\tau$-perpendicular, left finite wide, and right finite wide subcategories of $\modd\Lambda$ coincide. 
    \end{enumerate}
\end{theorem}

\section{$\tau$-exceptional sequences}\label{sec2}

Exceptional sequences were introduced in the representation theory of hereditary algebras by Crawley-Boevey \cite{Boevey-ExcSeq} and Ringel \cite{braid_group_action_on_exc_seq_Ringel}. When the length of such a sequence equals the number of simple modules, the sequence is said to be \textit{complete}. Complete exceptional sequences always exist over hereditary algebras, although this is not always the case for non-hereditary algebras. Buan and Marsh extended this concept by defining $\tau$-exceptional sequences \cite{tauExcSeq_BM} which generalize classical exceptional sequences. In particular, complete $\tau$-exceptional sequences always exist over an arbitrary finite dimensional algebra.

 In this section, we recall the definition of a ($\tau$-)exceptional sequence and review a bijection established in \cite[Section 5]{nonis2025tau} between ($\tau$-)exceptional sequences in $\modd{kQ}$ and $\tau$-exceptional sequences in $\modd{\Lambda}$, where $\Lambda = R\otimes kQ$. 

\begin{definition}[{\cite{Boevey-ExcSeq}}]
    Let $\Lambda$ be a finite dimensional algebra. For a positive integer $t$, an ordered $t$-tuple of indecomposable $\Lambda$-modules $(M_1,\cdots, M_t)$ in $\modd\Lambda$ is called \emph{exceptional} if the following conditions hold. 
    \begin{enumerate}
        \item[(a)] $\End_\Lambda(M_i)\cong k$ for $1\leq i\leq t$;
        \item[(b)]  $\Ext_\Lambda^{\geq 1}(M_i,M_i) = 0$ for $1\leq i \leq t$; 
        \item[(c)]  $\Hom_\Lambda(M_i,M_j) = 0 = \Ext_\Lambda^{\geq 1}(M_i,M_j)$ for $1\leq j < i \leq t$.  
    \end{enumerate}
    If $t = \abs\Lambda$, the sequence is said to be \emph{complete}.
\end{definition}

\begin{definition}[{\cite[Def. 1.3]{tauExcSeq_BM}}]
    Let $\Lambda$ be a finite dimensional algebra. For a positive integer $t$, an ordered $t$-tuple of indecomposable $\Lambda$-modules $(M_1,\cdots, M_t)$ is a \emph{$\tau$-exceptional sequence} if 
    \begin{enumerate}
        \item[(a)] $M_t$ is $\tau$-rigid in $\modd\Lambda$, and
        \item[(b)] $(M_1,\cdots, M_{t-1})$ is a $\tau$-exceptional sequence in $J(M_t)$.
    \end{enumerate}
    If $t=\abs{\Lambda}$, the sequence is said to be \emph{complete}.  
\end{definition}  

Recall that for $M, N\in\modd \Lambda$ there is a functorial isomorphism (AR duality)
\begin{equation*}
    \Ext_\Lambda^1(M,N) \cong D\overline{\Hom}_\Lambda(N,\tau M).
\end{equation*}

 In particular, if $\pd M \leq 1$ we have that 
 \begin{equation}\label{AR_Duality_pdM<=1}
      \Ext_\Lambda^1(M,N) \cong D\Hom_\Lambda(N,\tau M).
 \end{equation}

\begin{remark}\label{tau-exc=exc_over_hereditary}
    Let $\Lambda$ be a hereditary algebra. By the AR duality \eqref{AR_Duality_pdM<=1}, ${^\perp}(\tau X)$ consists of all the $\Lambda$-modules $Y$ such that $\Ext^1_\Lambda(X,Y) = 0$. Hence, given a $\tau$-rigid $\Lambda$-module $X$ we have 
    $J(X) = X^{\perp}\cap {^\perp \tau X} = X^{\perp_{0,1}}$, where $$X^{\perp_{0,1}} = \{Y\in\modd\Lambda \mid \Hom_\Lambda(X,Y) = 0 = \Ext^1_\Lambda(X,Y)\}.$$
    In other words, $\tau$-perpendicular and Geigle-Lenzing perpendicular subcategories \cite{Geigle_Lenzing_Perp_Subcat} of $\modd{\Lambda}$ coincide. As a result, $\tau$-exceptional and exceptional sequences also coincide over hereditary algebras; see \cite{tauExcSeq_BM}. 
\end{remark}

We recall the following uniqueness property of $\tau$-exceptional sequences. 

\begin{theorem}[{\cite[Thm. 8]{uniqueness_of_tau_exc_seq}}]\label{uniqueness_property_of_tau_exc_seq}
    Let $\Lambda$ be a finite dimensional algebra. Let $(X_1,\cdots,X_n)$ and $(Y_1,,\cdots, Y_n)$ be complete $\tau$-exceptional sequences in $\modd \Lambda$. If there exists an index $j$ such that $X_i\cong X_i$ for $i\neq j$, then also $X_j \cong X_j$. 
\end{theorem}

Now, let $\Lambda = R\otimes kQ$. Let $t$ be a positive integer and let $\M = (M_1,\cdots, M_t)$ be a ($\tau$-)exceptional sequence in $\modd{kQ}$. 
Then, 
\begin{equation}\label{ind_functor(exc_seq)}
    \Lambda\otimes_{kQ}\M := (\Lambda\otimes_{kQ}M_1,\cdots, \Lambda\otimes_{kQ}M_t)
\end{equation}
is a $\tau$-exceptional sequence in $\modd{\Lambda}$ by \cite[Prop. 5.8]{nonis2025tau}. In particular, every $\tau$-exceptional sequence in $\modd{\Lambda}$ arizes in this way. 

\begin{theorem}[{\cite[Thm. 5.10]{nonis2025tau}}]\label{bijection_of_tau_exc_seq}
    Let $t$ be a positive integer. Then, there is a bijection 
    $$\Lambda\otimes_{kQ}-: {\left\{ \begin{tabular}{@{}l@{}} ($\tau$-)exc. seq. in $\modd{kQ}$ \\ \multicolumn{1}{c}{of length $t$}\end{tabular} \right\}}\to{\left\{ \begin{tabular}{@{}l@{}} $\tau$-exc. seq. in $\modd{\Lambda}$ \\ \multicolumn{1}{c}{of length $t$}\end{tabular} \right\}}$$
    given by $(M_1,,\cdots,M_t)\mapsto(\Lambda\otimes_{kQ}M_1,\cdots, \Lambda\otimes_{kQ}M_t)$.
\end{theorem}

If $t = \abs{\Lambda}$, we obtain the following corollary. 

\begin{corollary}
    The induction functor induces a bijection between the set of complete ($\tau$-)exceptional sequences in $\modd{kQ}$ and the set of complete $\tau$-exeptional sequences in $\modd{\Lambda}$. 
\end{corollary}

\section{Mutation}\label{secMutation}

In this section, we discuss mutation of complete $\tau$-exceptional sequences in $\modd\Lambda$, where $\Lambda = R\otimes kQ$. 

In the hereditary case, there is a transitive action of the braid group on the set of complete exceptional sequences \cite{Boevey-ExcSeq, braid_group_action_on_exc_seq_Ringel}. More precisely, for a complete exceptional sequence $\E = (X_1,\cdots,X_n)$ in $\modd kQ$ and $1\leq i\leq n-1$, the left action of the braid group on $\E$ is given by 
$$\sigma_i(\E) = \sigma_i(X_1,\cdots, X_n) = (X_1,\cdots,X_{i-1}, Y, X_i, X_{i+2}, \cdots, X_n)$$
for a unique indecomposable exceptional $kQ$-module $Y$ such that $\sigma_i(\E)$ is a complete exceptional sequence (the $kQ$-module $Y$ is explicitly described in \cite{ThreeKindsOfMutations}). 

Recently, Buan, Hanson, and Marsh \cite{BHM_mutation} defined a notion of mutation for complete $\tau$-exceptional sequences over an arbitrary finite dimensional algebra (BHM mutation).  We refer to a ($\tau$-)exceptional sequence of length two as a \textit{($\tau$-)exceptional pair}. We recall their terminology. 

Let $\mathcal{S}$ be a subcategory of $\modd \Lambda$. We denote by $\P(\mathcal{S)}$ (respectively, $\P_s(\mathcal{S})$) the full subcategory of $\mathcal{S}$ consisting of Ext-projective (respectively, split Ext-projective) modules in $\mathcal{S}$. Additionally, $\P_{ns}(\mathcal{S})$ denotes the full subcategory of $\mathcal{S}$ consisting of all the modules in $\P(\mathcal{S})$ none of whose indecomposable direct summands are split Ext-projective in $\mathcal{S}$. 

Dually, we denote by $\mathcal{I}(\mathcal{S})$ (respectively, $\mathcal{I}_s(\mathcal{S})$) the subcategory of $\mathcal{S}$ consisting of the modules in $\mathcal{S}$ which are Ext-injective (resp. split Ext-injective) in $\mathcal{S}$. Additionally, $\mathcal{I}_{ns}(\mathcal{S})$ denotes the full subcategory of $\mathcal{S}$ consisting of all the modules in $\mathcal{I}(\mathcal{S})$ none of whose indecomposable direct summands are split Ext-injective in $\mathcal{S}$.

\begin{definition}[{\cite[Def. 3.11]{BHM_mutation}}\label{left_regular_pair}] Let $\Lambda$ be a finite dimensional algebra. 
    \begin{enumerate}
        \item[(a)] A $\tau$-exceptional pair $(B,C)$ with $C\notin \P({^\perp}\tau\mathcal{E}_C^{-1}(B))$ or $C\in \proj\Lambda$ is called \emph{left regular}. Otherwise, it is called \emph{left irregular}; 
        \item[(b)]  A $\tau$-exceptional pair $(X,Y)$ with $\mathcal{E}^{-1}_Y(X)\in \P(^{\perp}\tau Y)$ or $Y\notin \Gen \mathcal{E}^{-1}_Y(X)$ is called \emph{right regular}. Otherwise, it is called \emph{right irregular}.
    \end{enumerate}
\end{definition}

\begin{definition}[{\cite[Def. 4.1]{BHM_mutation}}]
    A $\tau$-exceptional pair $(B,C)$, is called \emph{left immutable} if it is left irregular and the $\tau$-perpendicular category $J(B,C)$ is not left finite. A $\tau$-exceptional pair $(B,C)$ is called \emph{left mutable} if it is not left immutable.  
\end{definition}

\begin{definition}[{\cite[Def. 4.2]{BHM_mutation}}]
    A $\tau$-exceptional pair $(X,Y)$, is called \emph{right immutable} if it is right irregular and the $\tau$-perpendicular category $J(B,C)$ is not right finite. A $\tau$-exceptional pair $(X,Y)$ is called \emph{right mutable} if it is not right immutable.  
\end{definition}

For an indecomposable object $U$ in $\C(\Lambda)$, we denote 
$$\abs{U}= 
\begin{cases}
    U & \text{if }U\in\modd\Lambda,\\
    U[-1] & \text{if }U\in\modd{\Lambda}[1]. 
\end{cases}$$

For a $\tau$-exceptional pair $(M,N)$, we denote 
\begin{align*}
    &N_{+} := \begin{cases} N[1] &\text{if } N \in\proj\Lambda \\ N &\text{otherwise},\end{cases} &&M^{+}:=\begin{cases} M[1] &\text{if }M\in\proj J(N)\\ M &\text{otherwise},\end{cases}\\
    &M_{N\uparrow}:= \mathcal{E}^{-1}_{N_+}(M), && M^{+}_{N\uparrow}:= \mathcal{E}^{-1}_N(M^+).
\end{align*}

\begin{defprop}[{\cite[Def.-Prop. 4.3]{BHM_mutation}}]\label{varphi_left_regular}
    For a left regular $\tau$-excepional pair $(B,C)$ let $$\varphi(B,C):=\left(\abs{\mathcal{E}_{B_{C\uparrow}}(C_+)}, B_{C\uparrow}\right).$$ 
    Then $\varphi(B,C)$ is a right regular $\tau$-exceptional pair with $J(\varphi(B,C)) = J(B,C)$. 
\end{defprop}

\begin{defprop}[{\cite[Def.-Prop. 1.1]{BHM_mutation}}]
    Let $\mathcal{S}$ be a subcategory of $\modd\Lambda$. The following statements hold. 
    \begin{enumerate}
        \item[(a)] $({^\perp}(\mathcal{S}^{\perp}), \mathcal{S}^{\perp})$ and $({^\perp}{\mathcal{S}}, ({^\perp}\mathcal{S})^\perp)$ are torsion pairs;
        \item[(b)]  The subcategory ${^\perp}(\mathcal{S}^{\perp})=\Filt(\Gen\mathcal{S})=:\T(\mathcal{S})$ is the smallest torsion class containing $\mathcal{S}$;
        \item[(c)]  The subcategory $({^\perp}\mathcal{S})^\perp=\Filt(\Cogen\mathcal{S})=:\F(\mathcal{S})$ is the smallest torsion-free class containing $\mathcal{S}$.  
    \end{enumerate}
\end{defprop}

\begin{defprop}[{\cite[Def.-Prop. 4.4]{BHM_mutation}}]\label{varphi_left_irregular}
    Let $(B,C)$ be a left mutable left irregular $\tau$-exceptional pair and let $L:= \mathcal{E}^{-1}_C(B)\oplus C$. Let $$X':= \P_s(\Gen \P_{ns}(\T(J(L))))  \quad\text{ and }\quad Y:= \P_{ns}(\T(J(L)))/X'.$$ Then $\varphi(B,C):=\left({\mathcal{E}_{Y}(X')}, Y\right)$ 
    is a right irregular right mutable $\tau$-exceptional pair with $J(\varphi(B,C)) = J(B,C) = J(L)$. 
\end{defprop}

With the terminology established above, the following theorem summarizes some of the main results of \cite{BHM_mutation}. 

 \begin{theorem}\label{main_thm_BHM}
      \begin{enumerate}
        \item[(i)] \cite[Thm. 4.7]{BHM_mutation} There is a bijection 
        $$\left\{ \text{left mutable }\tau\text{-exceptional pairs} \right\}\xrightarrow{\varphi} \left\{ \text{right mutable }\tau\text{-exceptional pairs} \right\};$$
        \item[(ii)] \cite[Thm. 4.11]{BHM_mutation} Suppose $\Lambda$ is hereditary, $\tau$-tilting finite, or has rank two. Then every $\tau$-exceptional pair is both left and right mutable. 
    \end{enumerate}
 \end{theorem}
 
The next observation gives a new class of examples for which every $\tau$-exceptional pair is both left and right mutable.

\begin{proposition}\label{left-right_mutable_pair_for_Lambda}
    Let $\Lambda = R\otimes kQ$. Then every $\tau$-exceptional pair is both left and right mutable. 
\end{proposition}

\begin{proof}
    As observed in the Proof of Theorem \cite[Thm. 4.11]{BHM_mutation}, it suffices to show that every $\tau$-perpendicular category is both left and right finite (see Definition \ref{left_regular_pair}). But $\tau$-perpendicular categories, left finite wide subcategories, and right finite wide subcategories of $\modd\Lambda$ coincide by Theorem \ref{bijection_of_tau-perp_subcategories} (i). The claim follows. 
\end{proof}

We have the following observation. 

\begin{proposition}\label{bijection_of_left_(ir)regular_pairs}
    There are bijections 
    \begin{equation}\label{(4)left_regular}
        \Lambda\otimes_{kQ}-: {\left\{ \begin{tabular}{@{}l@{}} left regular ($\tau$-)exc. \\ \quad pairs in $\modd{kQ}$ \end{tabular} \right\}}\to{\left\{ \begin{tabular}{@{}l@{}} left regular ($\tau$-)exc. \\ \quad  pairs in $\modd{\Lambda}$\end{tabular} \right\}}
    \end{equation}

    \begin{equation}\label{bijection_of_left_irr_pairs}
        \Lambda\otimes_{kQ}-: {\left\{ \begin{tabular}{@{}l@{}} left irregular ($\tau$-)exc. \\ \quad pairs in $\modd{kQ}$ \end{tabular} \right\}}\to{\left\{ \begin{tabular}{@{}l@{}} left irregular $\tau$-exc. \\ \quad  pairs in $\modd{\Lambda}$\end{tabular} \right\}}
    \end{equation}
    given by the induction functor. 
\end{proposition}

\begin{proof}
   We only prove the first bijection since the second one is similar. By Theorem \ref{bijection_of_tau_exc_seq}, the induction functor induces a bijection between ($\tau$-)exceptional pairs in $\modd{kQ}$ and $\tau$-exceptional pairs in $\modd\Lambda$, and therefore the assignment \eqref{(4)left_regular} is injective. 
   
   We are left to prove that $\Lambda\otimes_{kQ}-$ is surjective on left regular $\tau$-exceptional pairs. Let $(X,Y)$ be a left regular $\tau$-exceptional pair in $\modd \Lambda$. Then, $(X,Y)$ is isomorphic to $(\Lambda\otimes_{kQ}B, \Lambda\otimes_{kQ}C)$ for a ($\tau$-)exceptional pair $(B,C)$ in $\modd kQ$ by Theorem \ref{bijection_of_tau_exc_seq}. We show that $(B,C)$ is left regular. By definition of a left regular pair, $\Lambda\otimes_{kQ}C\notin \P({^\perp}\tau \Eps^{-1}_{\Lambda\otimes_{kQ}C}(\Lambda\otimes_{kQ}B))$ or $\Lambda\otimes_{kQ}C\in\proj\Lambda$. Since the induction functor induces a bijection between $\P({^\perp}\tau \Eps_C^{-1}B)$ and $\P({^\perp}\tau(\Lambda\otimes_{kQ}\Eps^{-1}_C B))$ (see Proposition \ref{indprojJ(M)1-1indprojJ(Lambda_otimes_M)}), and $\P({^\perp}\tau \Eps^{-1}_{\Lambda\otimes_{kQ}C}(\Lambda\otimes_{kQ}B)) = \P({^\perp}(\Lambda\otimes_{kQ}\tau\Eps^{-1}_C B))$ by Proposition \ref{E-version3.7} and \cite[Prop. 1.8(iii)]{nonis2025tau}, it follows that $\Lambda\otimes_{kQ}C\notin\P({^\perp}(\Lambda\otimes_{kQ}\tau\Eps^{-1}_C B))$ if and only if $C\notin \P(^{\perp}\tau \Eps_C^{-1}B)$. Moreover, $\Lambda\otimes_{kQ}C\in \proj\Lambda$ if and only if $C\in\proj{kQ}$. We conclude that $(B,C)$ is a left regular ($\tau$-)exceptional pair in $\modd{kQ}$, giving surjectivity. This finishes the proof. 
\end{proof}

Mutation of $\tau$-exceptional pairs induces a mutation on $\tau$-exceptional sequences as follows. Let $\Lambda$ be a finite dimensional algebra and let $\W\subseteq \modd\Lambda$ be a $\tau$-perpendicular subcategory of $\modd\Lambda$. We denote by $\varphi^{\W}$ the map $\varphi$ in the category $\W$. 

\begin{definition}[{\cite[Def. 5.1]{BHM_mutation}}]
    Let $\Lambda$ be a finite dimensional algebra and let $t$ be a positive integer. Let $\M = (M_1,\cdots, M_t)$ be a  $\tau$-exceptional sequence in $\modd\Lambda$ and let $i\in \{1,\cdots,t-1\}$. 
    \begin{enumerate}
        \item[(a)] We say that $\M$ is \emph{left $i$-mutable} (resp. \emph{right $i$-mutable}) if $(M_i,M_{i+1})$ is a left mutable (resp. right mutable) $\tau$-exceptional pair in $J(M_{i+2},\cdots, M_t)$;  
        \item[(b)] We say that $\M$ is \emph{left $i$-regular} (resp. \emph{right $i$-regular}) if $(M_i,M_{i+1})$ is a left regular (resp. right regular) $\tau$-exceptional pair in $J(M_{i+2},\cdots, M_t)$;   
         \item[(c)] We say that $\M$ is \emph{left $i$-irregular} (resp. \emph{right $i$-irregular}) if $(M_i,M_{i+1})$ is a left regular (resp. right irregular) $\tau$-exceptional pair in $J(M_{i+2},\cdots, M_t)$.   
    \end{enumerate}
\end{definition}

\begin{definition}[{\cite[Def. 5.2]{BHM_mutation}}]
    Let $\Lambda$ be a finite dimensional algebra and let $t$ be a positive integer. Let $\M = (M_1,\cdots, M_t)$ be a $\tau$-exceptional sequence in $\modd\Lambda$ and let $i\in \{1,\cdots,t-1\}$. If $\M$ is left $i$-mutable, then the $i$-\emph{th left mutation} of $\M$ is the $\tau$-exceptional sequence
    $$ \varphi_i(\M) := (M_1,\cdots, M_i',M_{i+1}',\cdots, M_t)$$
    where $(M_i',M_{i+1}') = \varphi^{J(M_{i+2},\cdots, M_t)}(M_i, M_{i+1})$;
    
\end{definition}

As a direct consequence of Theorem \ref{main_thm_BHM} (i), one obtains a bijection \cite[Cor. 5.3]{BHM_mutation}

$${\left\{ \begin{tabular}{@{}l@{}} left $i$-mutable  $\tau$-exceptional \\ \quad sequences of length $t$\end{tabular} \right\}}\xrightarrow{\varphi_i}{\left\{ \begin{tabular}{@{}l@{}} right $i$-mutable  $\tau$-exceptional \\ \quad sequences of length $t$\end{tabular} \right\}}.$$

An algebra $\Lambda$ is \textit{mutation complete} if every $\tau$-exceptional sequence (of length $t$) is both left and right $i$-mutable (for all $1\leq i \leq t-1$). 

\begin{corollary}[{\cite[Cor. 5.4]{BHM_mutation}}]
  Let $\Lambda$ be hereditary, $\tau$-tilting finite, or has rank two. Then $\Lambda$ is mutation complete.   
\end{corollary}

The following result gives a new class of algebras that are mutation complete.  

\begin{corollary}\label{Lambda_is_mutation_complete}
    Let $\Lambda = R\otimes kQ$. Then $\Lambda$ is mutation complete. 
\end{corollary}

\begin{proof}
    By Theorem \ref{J(Lambda_otimes_kQ)=mod(R_otimes_kQ')}, every $\tau$-perpendicular subcategory of $\modd\Lambda$ is equivalent to the module category of a finite dimensional algebra $R\otimes kQ'$ for an acyclic quiver $Q'$. The result follows applying Proposition \ref{left-right_mutable_pair_for_Lambda}. 
\end{proof}

\begin{example} Let $\Lambda = kQ/I$, where $Q$ is the quiver given by 
    
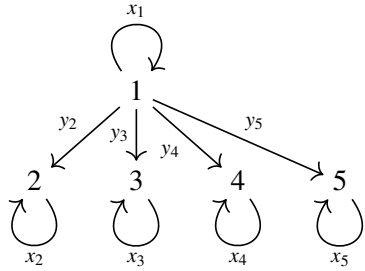
\[\begin{tikzcd}[ampersand replacement=\&,cramped]
	\& 1 \\
	2 \& 3 \& 4 \& 5
	\arrow["{x_1}", from=1-2, to=1-2, loop, in=55, out=125, distance=10mm]
	\arrow["{y_2}"', from=1-2, to=2-1]
	\arrow["{y_3}"', from=1-2, to=2-2]
	\arrow["{y_4}"', from=1-2, to=2-3]
	\arrow["{y_5}", from=1-2, to=2-4]
	\arrow["{x_2}", from=2-1, to=2-1, loop, in=235, out=305, distance=10mm]
	\arrow["{x_3}", from=2-2, to=2-2, loop, in=235, out=305, distance=10mm]
	\arrow["{x_4}", from=2-3, to=2-3, loop, in=235, out=305, distance=10mm]
	\arrow["{x_5}", from=2-4, to=2-4, loop, in=235, out=305, distance=10mm]
\end{tikzcd}\]
and $I$ is the ideal generated by the relations $x_i^2$ and $y_{j+1}x_1-x_{j+1}y_{j+1}$ for $1\leq i \leq 5$ and $1\leq j \leq 4$. Notice that $\Lambda\cong R\otimes kQ'$ where $Q'$ is the quiver obtained from $Q$ by removing the loop $x_i$ at each vertex $i$. Clearly, $\Lambda$ is not hereditary and has not rank two. Moreover, since $Q'$ is not of Dynkin type, $\Lambda$ is $\tau$-tilting infinite by \cite[Cor. 1.15]{nonis2025tau}. Nevertheless, $\Lambda$ is mutation complete by Corollary \ref{Lambda_is_mutation_complete}. 
\end{example}

If $\Lambda$ is hereditary, then the mutation of ($\tau$-)exceptional sequences coincides with the classical mutation of exceptional sequences. More precisely, we have the following result. 

\begin{theorem}\label{hereditary_mutation}
    Let $\Lambda$ be hereditary. 
    \begin{enumerate}
        \item[(i)] \cite[Prop. 6.2]{BHM_mutation} Let $(B,C)$ be an exceptional pair and write $\varphi(B,C) = (C',B')$. Then $B = B'$. 
        \item[(ii)] \cite[Thm. 6.3]{BHM_mutation} Let $\M$ be a complete ($\tau$-)exceptional sequence. Then $\varphi_i(\M) = \sigma_i(\M)$ for all $1\leq i \leq n-1$. That is, the mutation of ($\tau$-)exceptional sequences \cite{BHM_mutation} coincides with the braid group action \cite{Boevey-ExcSeq, braid_group_action_on_exc_seq_Ringel} on exceptional sequences. 
    \end{enumerate}
\end{theorem}

This section aims to prove the following analog of Theorem \cite[Thm. 6.3]{BHM_mutation} for $\Lambda = R\otimes kQ$. 

\begin{theorem}\label{mutation_in_Lambda}
    Let $\Lambda = R\otimes kQ$ and let $\Lambda\otimes_{kQ}\M$ be a complete $\tau$-exceptional sequence. Then $$\varphi_i(\Lambda\otimes_{kQ}\M) = \Lambda\otimes_{kQ}\sigma_i(\M)$$ for all $1\leq i \leq n-1$. 
\end{theorem}

The strategy to prove Theorem \ref{mutation_in_Lambda} is similar to the one used to prove \cite[Thm. 6.3]{BHM_mutation}. Indeed, Theorem \ref{mutation_in_Lambda} can be deduced from the following analog of \cite[Prop. 6.2]{BHM_mutation}. 

\begin{proposition}\label{mutation_pairs_in_Lambda}
    Let $\Lambda = R\otimes kQ$. Let $(\Lambda\otimes_{kQ}B,\Lambda\otimes_{kQ}C)$ be a $\tau$-exceptional pair and write $\varphi(\Lambda\otimes_{kQ}B,\Lambda\otimes_{kQ}C) = ((\Lambda\otimes_{kQ}C)',(\Lambda\otimes_{kQ}B)')$. Then $(\Lambda\otimes_{kQ}B)'=\Lambda\otimes_{kQ}B$.
\end{proposition}

\begin{proof}[Proposition \ref{mutation_pairs_in_Lambda} implies Theorem \ref{mutation_in_Lambda}]
    Let $\widetilde{\M}$ be a complete $\tau$-exceptional sequence in $\modd\Lambda$. Then $\widetilde{\M} = \Lambda\otimes_{kQ}\M$ for a unique ($\tau$-)exceptional sequence $\M = (M_1,\cdots, M_n)$ in $\modd kQ$ by Theorem \ref{bijection_of_tau_exc_seq}. Let $i\in\{1,\cdots, n-1\}$. Since $J(M_{i+1},\cdots, M_n)\simeq \modd(R\otimes kQ')$ for an acyclic quiver $Q'$ by Theorem \ref{J(Lambda_otimes_kQ)=mod(R_otimes_kQ')}, we have that 
    \begin{align*}
        \varphi_i(\widetilde{\M}) &= \varphi_i(\Lambda\otimes_{kQ}\M)\\
                                  &= (\Lambda\otimes_{kQ}M_1,\cdots, \Lambda\otimes_{kQ}M_{i-1}, \widetilde{X}, \Lambda\otimes_{kQ}M_{i},\cdots, \Lambda\otimes_{kQ}M_n) &&\text{by Prop. \ref{mutation_pairs_in_Lambda}}\\
                                  &= (\Lambda\otimes_{kQ}M_1,\cdots, \Lambda\otimes_{kQ}M_{i-1}, \Lambda\otimes_{kQ}X, \Lambda\otimes_{kQ}M_{i},\cdots, \Lambda\otimes_{kQ}M_n) &&\text{by Thm. \ref{bijection_of_tau_exc_seq}}.
    \end{align*}
    
    On the other hand, $\varphi_i(\M)=(M_1,\cdots,M_{i-1}, X', M_i,\cdots, M_n)$ by Theorem \ref{hereditary_mutation}(ii). Applying the induction functor to $\varphi_i(\M)$ we obtain a $\tau$-exceptional sequence 
    $$\Lambda\otimes_{kQ}\varphi_i(\M) = (\Lambda\otimes_{kQ}M_1,\cdots,\Lambda\otimes_{kQ}M_{i-1},\Lambda\otimes_{kQ}X',\Lambda\otimes_{kQ}M_i,\cdots,\Lambda\otimes_{kQ}M_n)$$ in $\bmod\Lambda$. Theorem \ref{uniqueness_property_of_tau_exc_seq} implies that $\Lambda\otimes_{kQ}X\cong \Lambda\otimes_{kQ}X'$. The result follows.    
\end{proof}

For the rest of this section, let $\Lambda = R\otimes kQ$. Following the argument in \cite[Section 6]{BHM_mutation}, we prove Proposition \ref{mutation_pairs_in_Lambda} in several steps. Notice that every $\tau$-exceptional pair $(\Lambda\otimes_{kQ}B, \Lambda\otimes_{kQ}C)$ in $\modd\Lambda$ is mutable by Corollary \ref{Lambda_is_mutation_complete}. We start by proving the claim for the case where $(\Lambda\otimes_{kQ}B, \Lambda\otimes_{kQ}C)$ is left regular. We need the following observation first. 

\begin{lemma}\label{Gen_Lemma}
    Let $\Lambda = R\otimes kQ$ and let $M,N\in \modd kQ$. Then, $\Lambda\otimes_{kQ}N\in \Gen(\Lambda\otimes_{kQ}M)$ if and only if $N\in \Gen M$. 
\end{lemma}

\begin{proof}
    Let $\Lambda\otimes_{kQ}N\in \Gen(\Lambda\otimes_{kQ}M)$. Then, there exists a surjective map $(\Lambda\otimes_{kQ}M)^r\cong \Lambda\otimes_{kQ}M^r\to \Lambda\otimes_{kQ}N$, for some $r\geq 1$. Applying the restriction of scalars, we obtain a surjective map $M^{rd}\to N$, where $d=\dim R$. In other words, $N$ lies in $\Gen M$. Conversely, let $N\in \Gen M$, i.e. there exists an epimorphism $M^s\to N$ for some $s\geq 1$. Since the induction functor is exact, we get an epimorphism $\Lambda\otimes_{kQ}M^s \cong (\Lambda\otimes_{kQ}M)^s\to \Lambda\otimes_{kQ}N$, and therefore $\Lambda\otimes_{kQ}N\in \Gen(\Lambda\otimes_{kQ}M)$.  
\end{proof}

\begin{proposition}\label{mutation_left_regular}
    Let $(\Lambda\otimes_{kQ}B, \Lambda\otimes_{kQ}C)$ be a left regular ($\tau$-)exceptional pair. Then $\varphi(\Lambda\otimes_{kQ}B,\Lambda\otimes_{kQ}C) = ((\Lambda\otimes_{kQ}C)', \Lambda\otimes_{kQ}B)$. 
\end{proposition}

\begin{proof}
    Recall from Definition-Proposition \ref{varphi_left_regular} that $$(\Lambda\otimes_{kQ}B)' = (\Lambda\otimes_{kQ}B)_{(\Lambda\otimes_{kQ}C)\uparrow} = \mathcal{E}^{-1}_{(\Lambda\otimes_{kQ}C)_+}(\Lambda\otimes_{kQ}B).$$
    We distinguish three cases. 
    \begin{enumerate}
        \item[(1)] $\Lambda\otimes_{kQ}C\in \proj\Lambda$. Then $(\Lambda\otimes_{kQ}B)_{(\Lambda\otimes_{kQ}C)\uparrow} = \mathcal{E}^{-1}_{(\Lambda\otimes_{kQ}C)[1]}(\Lambda\otimes_{kQ}B) = \Lambda\otimes_{kQ}B$; see \cite[Thm. 2.8]{BHM_mutation}. 
        
         \item[(2)] $\Lambda\otimes_{kQ}C\in \Gen\mathcal{E}^{-1}_{\Lambda\otimes_{kQ}C}(\Lambda\otimes_{kQ}B)$.  By Proposition \ref{E-version3.7}, $\Eps^{-1}_{\Lambda\otimes_{kQ}C}(\Lambda\otimes_{kQ}B)\cong \Lambda\otimes_{kQ}\Eps^{-1}_C(B)$. Hence, $\Gen\mathcal{E}^{-1}_{\Lambda\otimes_{kQ}C}(\Lambda\otimes_{kQ}B) = \Gen(\Lambda\otimes_{kQ}\Eps_C^{-1}(B))$, and thus $\Lambda\otimes_{kQ}C\in \Gen\Eps^{-1}_{\Lambda\otimes_{kQ}C}(\Lambda\otimes_{kQ}B)$ implies $C\in \Gen\mathcal{E}^{-1}_{C}(B)$ by Lemma \ref{Gen_Lemma}. Then, 
            $$(\Lambda\otimes_{kQ}B)' = \Eps^{-1}_{\Lambda\otimes_{kQ}C}(\Lambda\otimes_{kQ}B) \cong \Lambda\otimes_{kQ}\Eps^{-1}_C(B) = \Lambda\otimes_{kQ}B$$
        where the last equality follows from \cite[Lemma 6.3]{BHM_mutation}.         
         \item[(3)]  $\Lambda\otimes_{kQ}C\notin \proj\Lambda$ and $\Lambda\otimes_{kQ}C\notin \Gen\mathcal{E}^{-1}_{\Lambda\otimes_{kQ}C}(\Lambda\otimes_{kQ}B)$. Then $(\Lambda\otimes_{kQ}C)_+ = \Lambda\otimes_{kQ}C$. Hence, 
        \begin{align*}
            \mathcal{E}^{-1}_{\Lambda\otimes_{kQ}C}(\Lambda\otimes_{kQ}B) &= f^{-1}_{\Lambda\otimes_{kQ}C}(\Lambda\otimes_{kQ}B) && \text{by Thm. \ref{BM_bijection}(c)}\\
            & =\Lambda\otimes_{kQ}f^{-1}_C(B) && \text{by Prop. \ref{E-version3.7}}\\
            & =\Lambda\otimes_{kQ}B &&\text{by \cite[Prop. 6.2]{BHM_mutation}}.
        \end{align*}
    \end{enumerate}
    This finishes the proof.
\end{proof}

We are left to prove Proposition \ref{mutation_pairs_in_Lambda} for $(\Lambda\otimes_{kQ}B,\Lambda\otimes_{kQ}C)$ being left irregular.

Recall that a module $M$ is called \textit{gen-minimal} if, for any proper direct summand $M'$ of $M$, we have $\Gen M'\subsetneq \Gen M$. \textit{Cogen-minimal} modules are defined dually. Moreover, the assignment $\T\mapsto \P_s(\T)$ defines a bijection between functorially finite torsion classes and gen-minimal $\tau$-rigid modules; see \cite[Thm. 1.6]{BHM_mutation}. We have the following observation.

\begin{lemma}\label{gen-min-tau-rigid}
    Let $U$ be a basic, $\tau$-rigid $kQ$-module. Then, the induction functor $\Lambda\otimes_{kQ}-$ induces a bijection between the set of isoclasses of gen-minimal $\tau$-rigid modules in $J(U)\subseteq\modd kQ$ and the set of isoclasses of gen-minimal $\tau$-rigid modules in $J(\Lambda\otimes_{kQ}U)\subseteq \modd \Lambda$. 
\end{lemma}

\begin{proof}
    By Theorem \ref{J(Lambda_otimes_kQ)=mod(R_otimes_kQ')} and \cite[Cor. 4.11]{nonis2025tau}, it suffices to prove the claim for $J(U) = \modd kQ$ and $J(\Lambda\otimes_{kQ}U) = \modd \Lambda$. Consider the following diagram 
    \[\begin{tikzcd}[ampersand replacement=\&,cramped]
	{\stautilt kQ} \&\& {\taugenmin kQ} \\
	{\stautilt \Lambda} \&\& {\taugenmin \Lambda}
	\arrow["{{\P_s}(\Gen-)}", from=1-1, to=1-3]
	\arrow["{\Lambda\otimes_{kQ}-}"', from=1-1, to=2-1]
	\arrow["{\Lambda\otimes_{kQ}-}", from=1-3, to=2-3]
	\arrow["{{\P_s}(\Gen-)}", from=2-1, to=2-3]
\end{tikzcd}\]
where $\taugenmin kQ$ and $\taugenmin\Lambda$ denote the set of gen-minimal $\tau$-rigid $kQ$-modules and the set of gen-minimal $\tau$-rigid $\Lambda$-modules, respectively. Note that the left vertical arrow and the horizontal arrows are bijections by \cite[Cor. 1.14]{nonis2025tau} and \cite[Thm. 1.6]{BHM_mutation}, respectively. We claim that the diagram above commutes. Let $M\in\stautilt kQ$ and write $M = M_s\oplus M_{ns}$ as in \cite[Lemma 2.6]{tau-perpendicular_wide_subategories}, where $M_s$ is split projective in $\Gen M$ and no direct summand of $M_{ns}$ is split projective in $\Gen M$. Then,
\begin{align*}
   \Lambda\otimes_{kQ} {\P_s(\Gen M)} &= \Lambda\otimes_{kQ}M_s \\
   &= (\Lambda\otimes_{kQ}M)_s && \text{by \cite[Lemma 6.22]{nonis2025tau}}\\
   &= {\P_s}(\Gen(\Lambda\otimes_{kQ}M)).
\end{align*}
Since the diagram is commutative and the left vertical arrow, as well as the horizontal arrows, are bijections, it follows that the right vertical arrow is also a bijection. This completes the proof.
\end{proof}

 We need the following analog of \cite[Prop. 6.7]{BHM_mutation}.

\begin{proposition}\label{analog_of_6.7_BHM}
    Let $(B, C)$ be a ($\tau$-)exceptional pair in $\modd kQ$ and suppose that $\Ext^1_\Lambda(B,C) \neq 0$, and $C\notin\proj kQ$. Let $\eta = (C\hookrightarrow E\twoheadrightarrow B^r)$ and $\eta' = (C^l\hookrightarrow E'\twoheadrightarrow B)$ be short exact sequences such that $\eta: B^r\to C[1]$ is a minimal right $\add B$-approximation and $\eta': B\to C^l[1]$ is a minimal left $\add C[1]$-approximation. The following statements hold: 
    \begin{enumerate}
        \item[(i)] $\mathcal{E}^{-1}_{\Lambda\otimes_{kQ}C}(\Lambda\otimes_{kQ}B) = \Lambda\otimes_{kQ}E'$; 
        \item[(ii)] $\Hom_\Lambda(\Lambda\otimes_{kQ}E', \Lambda\otimes_{kQ}C) = 0$ and $\Lambda\otimes_{kQ}L:= \Lambda\otimes_{kQ}(C\oplus E')$ is a gen-minimal $\tau$-rigid module; 
        \item[(iii)] $\Lambda\otimes_{kQ}\widetilde{L}:= \Lambda\otimes_{kQ}(B\oplus E)$ satisfies $\Lambda\otimes_{kQ}\widetilde{L}\in \Gen(\Lambda\otimes_{kQ}L)$ and $\tau(\Lambda\otimes_{kQ}\widetilde{L})\in J(\Lambda\otimes_{kQ}L)^\perp$; 
        \item[(iv)] $\Hom_\Lambda(\Lambda\otimes_{kQ}B,\Lambda\otimes_{kQ}E)=0$ and $\tau(\Lambda\otimes_{kQ}\widetilde{L})$ is a cogen-minimal $\tau$-rigid module; 
        \item[(v)] $\mathcal{E}_{\Lambda\otimes_{kQ}E'}(\Lambda\otimes_{kQ}C)\in\proj J(\Lambda\otimes_{kQ}E')$; 
        \item[(vi)] $\Lambda\otimes_{kQ}\widetilde{L} = \P_{ns}(\T(J(\Lambda\otimes_{kQ}L)))$.  
    \end{enumerate}
\end{proposition}

\begin{proof}
     Observe that \cite[Prop. 6.7]{BHM_mutation} applies to the exceptional pair $(B,C)$. Moreover, the assumptions $\Ext^1_{kQ}(B,C) \neq 0$ and $C\notin \proj kQ$ imply $\Ext^1_\Lambda(\Lambda\otimes_{kQ}B,\Lambda\otimes_{kQ}C) \neq 0$ and $\Lambda\otimes_{kQ}C \notin \proj \Lambda$, respectively. 

    (i) Since $E'$ is an indecomposable $\tau$-rigid $kQ$-module and $\Eps^{-1}_C(B) = E'$ by \cite[Prop. 6.7(a)]{BHM_mutation}, it follows that  $\mathcal{E}^{-1}_{\Lambda\otimes_{kQ}C}(\Lambda\otimes_{kQ}B) = \Lambda\otimes_{kQ}\Eps^{-1}_C(B) = \Lambda\otimes_{kQ}E'$ by Proposition \ref{E-version3.7}.

    (ii) We have that $$\Hom_\Lambda(\Lambda\otimes_{kQ}E', \Lambda\otimes_{kQ}C) \cong \Hom_{R}(R,R)\otimes\Hom_{kQ}(E',C) = 0$$ since $\Hom_{kQ}(E',C) = 0$ by \cite[Prop. 6.7(b)]{BHM_mutation}.

    Now let $\Lambda\otimes_{kQ}L:= \Lambda\otimes_{kQ}(C\oplus E')$. We have that $\Lambda\otimes_{kQ}E'$ is indecomposable, $\Hom_\Lambda(\Lambda\otimes_{kQ}E',\Lambda\otimes_{kQ}C) = 0$, and $\Lambda\otimes_{kQ}L$ is a basic $\tau$-rigid $\Lambda$-module. To prove gen-minimality of $\Lambda\otimes_{kQ}L$ we show that $\Lambda\otimes_{kQ}E'\notin \Gen(\Lambda\otimes_{kQ}C)$. But this follows from the fact that $E'\notin \Gen C$ by \cite[Prop. 6.7(b)]{BHM_mutation} and Lemma \ref{Gen_Lemma}.

    (iii) Let $\Lambda\otimes_{kQ}\widetilde{L} = \Lambda\otimes_{kQ}(B\oplus E)$. By \cite[Prop. 6.7(c)]{BHM_mutation} $\widetilde{L} = B\oplus E$ is in $\Gen L$ and therefore $\Lambda\otimes_{kQ}\widetilde{L}\in \Gen(\Lambda\otimes_{kQ}L)$ by Lemma \ref{Gen_Lemma}.
    
    Now we show that $\tau(\Lambda\otimes_{kQ}\widetilde{L})\in J(\Lambda\otimes_{kQ}L)^\perp$. So let $X\in J(\Lambda\otimes_{kQ}L)$. Then, $\Hom_{kQ}(\Lambda\otimes_{kQ}C,X) = 0$ and $\Hom_{kQ}(\Lambda\otimes_{kQ}E',X)=0$. Moreover, the fact that $X$ is in $J(\Lambda\otimes_{kQ}L)$ together with $\pd(\Lambda\otimes_{kQ}L)\leq 1$ give $\Ext^1_{\Lambda}(\Lambda\otimes_{kQ}C,X) = 0$ and $\Ext_\Lambda^1(\Lambda\otimes_{kQ}E',X) = 0$. Applying $\Hom_\Lambda(-,X)$ to $\eta$ and $\eta'$ we obtain two exact sequences of the form 
    $$\Ext^1_\Lambda((\Lambda\otimes_{kQ}B)^r,X)\to \Ext_\Lambda^1(\Lambda\otimes_{kQ}E,X)\to \Ext^1_\Lambda(\Lambda\otimes_{kQ}C,X)$$
    and 
    $$\Hom_{\Lambda}((\Lambda\otimes_{kQ}C)^l, X)\to \Ext^1_\Lambda(\Lambda\otimes_{kQ}B,X)\to \Ext_\Lambda^1(\Lambda\otimes_{kQ}E',X).$$
    Since the outer terms of both sequences vanish, and using that $\pd(\Lambda\otimes_{kQ}\widetilde{L})\leq 1$, it follows that
    \begin{align*}
        0 = \Ext^1_\Lambda(\Lambda\otimes_{kQ}(B\oplus E), X) &= \Ext^1_\Lambda(\Lambda\otimes_{kQ}\widetilde{L},X)\\&\cong\Hom_\Lambda(X, \tau(\Lambda\otimes_{kQ}\widetilde{L})).
    \end{align*}
    This concludes the proof of (iii). 

    (iv) Since $\Hom_{kQ}(B,E)=0$ by \cite[Prop. 6.7(d)]{BHM_mutation}, it follows that $\Hom_\Lambda(\Lambda\otimes_{kQ}B, \Lambda\otimes_{kQ}E)=0$. Moreover, since $\tau\widetilde{L}$ is $\tau$-rigid, then so is $\Lambda\otimes_{kQ}\tau\widetilde{L} \cong \tau(\Lambda\otimes_{kQ}\widetilde{L})$ by Proposition \ref{tau-rigid-Lambda1:1tau-rigid-kQ} and \cite[Prop. 1.8(iii)]{nonis2025tau}. It remains to show that $\tau(\Lambda\otimes_{kQ}\widetilde{L})$ is cogen-minimal. To see this, it suffices to show that $\tau(\Lambda\otimes_{kQ}E)\notin \Cogen(\tau(\Lambda\otimes_{kQ}B))$. Suppose that was not the case. Then we would have a monomorphism $\tau(\Lambda\otimes_{kQ}B)^s\to \tau(\Lambda\otimes_{kQ}E)$, for some $s>0$. Notice that this can be written as $\Lambda\otimes_{kQ}\tau B^s\to \Lambda\otimes_{kQ}\tau E$ by \cite[Prop. 1.8(iii)]{nonis2025tau}. Applying the restriction of scalars, we obtain a monomorphism $\tau B^{ds}\to \tau E^d$, where $d= \dim R$, that is $\tau B\in\Cogen \tau E$. But this contradicts the cogen-minimality of $\tau \widetilde{L}$ in \cite[Prop. 6.7(d)]{BHM_mutation}. 

    (v) By \cite[Prop. 6.7(e)]{BHM_mutation} $\mathcal{E}_{E'}(C)=C\in \proj J(E')$. Then the claim follows combining Proposition \ref{E-version3.7} and Proposition \ref{indprojJ(M)1-1indprojJ(Lambda_otimes_M)}.
    
    (vi) We follow the argument in \cite[Proof of Prop. 6.7(f)]{BHM_mutation}.  By definition of $\Lambda\otimes_{kQ}E$ and $\Lambda\otimes_{kQ}E'$, it follows that $J(\Lambda\otimes_{kQ}L)=J(\Lambda\otimes_{kQ}\widetilde{L}) = J^d(\tau(\Lambda\otimes_{kQ}\widetilde{L}))$. By (iv), \cite[Prop. 1.8]{nonis2025tau}, and \cite[Prop. 2.13]{BHM_mutation} (taking $\F = \Cogen(\Lambda\otimes_{kQ}\tau\widetilde{L})$), we have 
    \begin{equation}\label{eq}
        \tau(\Lambda\otimes_{kQ}\widetilde{L})  = \mathcal{I}_s(J^d(\Lambda\otimes_{kQ}\tau\widetilde{L})^\perp) = \mathcal{I}_s(J(\Lambda\otimes_{kQ}L)^\perp).
    \end{equation}
    The fact that $B= f_C E'\notin \proj J(C)$ (see Proof of \cite[Prop. 6.7(f)]{BHM_mutation}) gives $\Lambda\otimes_{kQ}B = \Lambda\otimes_{kQ}f_C E'\notin \proj J(\Lambda\otimes_{kQ}C)$, as every module in $\proj J(\Lambda\otimes_{kQ}C)$ is isomorphic to $\Lambda\otimes_{kQ}X$ for some module $X\in \proj J(C)$ by Proposition \ref{indprojJ(M)1-1indprojJ(Lambda_otimes_M)}. Since $\Lambda\otimes_{kQ}C\notin\proj\Lambda$, \cite[Prop. 3.9]{BHM_mutation} implies that 
    \begin{align*}
     J(\Lambda\otimes_{kQ}L) &= J((\Lambda\otimes_{kQ}C)\oplus \Eps_{\Lambda\otimes_{kQ}C}^{-1}(\Lambda\otimes_{kQ}B)) &&\text{ by (i), (ii)}\\
     &= J_{J(\Lambda\otimes_{kQ}C)}\Big(\Eps_{\Lambda\otimes_{kQ}C}(\Eps_{\Lambda\otimes_{kQ}C}^{-1}(\Lambda\otimes_{kQ}B))\Big) &&\text{ by \cite[Thm. 6.4]{tau-perpendicular_wide_subategories}}\\
     &=J_{J(\Lambda\otimes_{kQ}C)}(\Lambda\otimes_{kQ}B)\\
     &=J(\Lambda\otimes_{kQ}B,\Lambda\otimes_{kQ}C)
    \end{align*}
    is sincere. We also notice that $J(\Lambda\otimes_{kQ}L)$ is left finite by Theorem \ref{bijection_of_tau-perp_subcategories}(ii). Applying $\tau^{-1}$ to \eqref{eq} and using \cite[Prop. 3.10(a')]{BHM_mutation}, we obtain $\P_{ns}(\T(J(\Lambda\otimes_{kQ}L))) = \tau^{-1}\mathcal{I}_s(J(\Lambda\otimes_{kQ}L)^{\perp}) = \Lambda\otimes_{kQ}\widetilde{L}$.
    \end{proof}

    We are now ready to prove Proposition \ref{mutation_pairs_in_Lambda}.

    \begin{proof}[Proof of Proposition \ref{mutation_pairs_in_Lambda}]
        Let $(\Lambda\otimes_{kQ}B, \Lambda\otimes_{kQ}C)$ be a left mutable $\tau$-exceptional pair and write $\varphi(\Lambda\otimes_{kQ}B,\Lambda\otimes_{kQ}C)=((\Lambda\otimes_{kQ}C)', (\Lambda\otimes_{kQ}B)')$. Assume $(\Lambda\otimes_{kQ}B, \Lambda\otimes_{kQ}C)$ is left regular. Then $(\Lambda\otimes_{kQ}B)' = \Lambda\otimes_{kQ}B$ by Proposition \ref{mutation_left_regular}. 

        Now suppose that $(\Lambda\otimes_{kQ}B, \Lambda\otimes_{kQ}C)$ is left irregular. Notice, that this corresponds bijectively to the left irregular pair $(B, C)$ in $\modd{kQ}$ by Proposition \ref{bijection_of_left_(ir)regular_pairs} \eqref{bijection_of_left_irr_pairs}.  As shown in the Proof of \cite[Prop. 6.2]{BHM_mutation}, we have that $\Ext^1_{kQ}(B,C) \neq 0$, and therefore $\Ext^1_\Lambda(\Lambda\otimes_{kQ}B, \Lambda\otimes_{kQ}C)\neq 0$. Denote 
        $$ \Lambda\otimes_{kQ}L := \Lambda\otimes_{kQ}(C\oplus \Eps^{-1}_C(B)) \quad \text{and}\quad
        \Lambda\otimes_{kQ}\widetilde{L}:= \P_{ns}(\T(J(\Lambda\otimes_{kQ}L))).$$
        Notice that these are precisely the modules $\Lambda\otimes_{kQ}L$ and $\Lambda\otimes_{kQ}\widetilde{L}$ from Proposition \ref{analog_of_6.7_BHM}. In particular,  Proposition \ref{analog_of_6.7_BHM} (iii) gives a decomposition $\Lambda\otimes_{kQ}\widetilde{L} = \Lambda\otimes_{kQ}(B\oplus E)$, with $\Lambda\otimes_{kQ}B\in \Gen(\Lambda\otimes_{kQ}E)$ since $B\in \Gen E$. Now, \cite[Prop. 3.13(g)]{BHM_mutation} together with the fact that $\Lambda\otimes_{kQ}B\in \Gen(\Lambda\otimes_{kQ}E)$ imply that $\Lambda\otimes_{kQ}E = {\P_s}{(\Gen \P_{ns}(\T(J(\Lambda\otimes_{kQ}L))))}$. Using Def.-Prop. \ref{varphi_left_irregular}, we obtain that $(\Lambda\otimes_{kQ}B)' = (\Lambda\otimes_{kQ}\widetilde{L})/(\Lambda\otimes_{kQ}E) = \Lambda\otimes_{kQ}B$. This finishes the proof. 
    \end{proof}

\begin{corollary}\label{transitive_action_on_tau-exc_seq_over_Lambda}
    The braid group acts transitively on the set of complete $\tau$-exceptional sequences in $\modd\Lambda$. 
\end{corollary}

\begin{proof}
    Let $\M$ be a complete ($\tau$-)exceptional sequence in $\modd{kQ}$. By Theorem \ref{mutation_in_Lambda} we have $\Lambda\otimes_{kQ}\sigma_i\M = \varphi_i(\Lambda\otimes_{kQ}\M)$. Since $\Lambda\otimes_{kQ}-$ induces a bijection between complete ($\tau$-)exceptional sequences in $\modd{kQ}$ and complete $\tau$-exceptional sequences in $\modd{\Lambda}$ by Theorem \ref{bijection_of_tau_exc_seq}, and the braid group acts transitively on the set of complete ($\tau$-)exceptional sequences in $\modd{kQ}$, we conclude that the action of the braid group on the set of $\tau$-exceptional sequences in $\modd\Lambda$ is transitive. 
\end{proof}

\section{$RQ$-lattices}\label{secRQ-lattices}

Let $\Lambda  = R\otimes kQ$. We know that $\Lambda$ is isomorphic to $RQ$ and we will henceforth identify $\Lambda$ with $RQ$. This section aims to compare $\tau$-exceptional sequences in $\modd\Lambda$ with the notion of exceptional sequences of $RQ$-lattices introduced by Crawley-Boevey in \cite{CB_RigidIntegralRep}.

\begin{definition}[{\cite[Page 375]{CB_RigidIntegralRep}}]\label{RQ-lattices}
    A $\Lambda$-module $X$ is called an \emph{$RQ$-lattice} if it is finitely generated and projective as an $R$-module; the $R$-module associated to the vertex i is then $X_i:= \overline{e_i}X$, where $\overline{e_i}$ is the idempotent given by the trivial path at $i$ in $RQ$. An $RQ$-lattice $X$ is called \emph{$R$-exceptional} if it is rigid and $\End_\Lambda(X)\cong R$. 
\end{definition}

\begin{definition}[{\cite[Page 383]{CB_RigidIntegralRep}}]\label{R-exc_seq}
    A sequence of indecomposable $RQ$-lattices $(X_1\cdots, X_r)$ is called \emph{$R$-exceptional} if the following conditions hold. 
    \begin{enumerate}
        \item[(a)] $X_i$ is $R$-exceptional for all $1\leq i \leq r$; 
        \item[(b)] $\Hom_\Lambda(X_i, X_j) = 0 = \Ext_\Lambda^{\geq 1} (X_i, X_j)$ for all $1\leq j < i \leq r$.  
    \end{enumerate}
    If $r = \abs{\Lambda}$, the sequence is said to be \emph{complete}. 
\end{definition}

\begin{remark}
    Let $X\in \modd \Lambda$ be an $RQ$-lattice. Then, $\pd X \leq 1$ by \cite[Lemma 3.1(a)]{CB_RigidIntegralRep}. Hence, condition (b) in Definition \ref{R-exc_seq} is equivalent to $\Hom_\Lambda(X_i, X_j) = 0 = \Ext_\Lambda^{1} (X_i, X_j)$ for all $1\leq j < i \leq r$.
\end{remark}

The following observation holds for any finite dimensional algebra $\Lambda$, and generalizes the fact that exceptional and $\tau$-exceptional sequences coincide over hereditary algebras. Although this result was stated in \cite{tauExcSeq_BM} without a proof, we include one here for the sake of completeness.

\begin{lemma}\label{exceptional_with_pd<=1_is_tau-exc}
    Let $\Lambda$ be a finite dimensional algebra. Let $(M_1,\cdots, M_r)$ be a sequence of indecomposable $\Lambda$-modules satisfying the following conditions. 
    \begin{enumerate}
        \item[(i)] $\pd M_i \leq 1$ for $1\leq i \leq r$; 
        \item[(ii)] $\Ext_\Lambda^1(M_i,M_i) = 0$ for $1\leq i\leq r$; 
        \item[(iii)] $\Hom_\Lambda(M_i,M_j) = 0 = \Ext_\Lambda^1(M_i,M_j)$ for $1\leq j < i \leq r$. 
    \end{enumerate}
    Then, $(M_1,\cdots, M_r)$ is a $\tau$-exceptional sequence. 
\end{lemma}

\begin{proof}
    We proceed by induction on the length of the sequence. If $r=1$, using the AR-duality \eqref{AR_Duality_pdM<=1}, we have that $0 = \Ext^1_\Lambda(M_1,M_1) \cong \Hom_\Lambda(M_1,\tau M_1)$. Hence, $M_1$ is $\tau$-rigid, giving the base case. 

    Now, let $i\in\{1,\cdots,r\}$ and assume that $(M_i,\cdots, M_r)$ is a $\tau$-exceptional sequence. We want to show that $(M_{i-1},\cdots,M_r)$ is also a $\tau$-exceptional sequence. In other words, we have to prove that $M_{i-1}$ is $\tau$-rigid in $J(M_i,\dots, M_r)$. Note that, using Proposition \ref{pd_preserved} together with the AR-duality \eqref{AR_Duality_pdM<=1}, and the fact that $\tau$-perpendicular categories are wide subcategories of $\modd \Lambda$, we can rewrite $J(M_i,\cdots, M_r)$ as 
    $$J(M_i,\cdots, M_r) = \{X\in\modd\Lambda \mid \Hom_\Lambda(M, X) = 0 = \Ext_\Lambda^1(M,X)\},$$
    where $M = \bigoplus_{j=i}^r M_j$. Hence, $M_{i-1}$ lies in $J(M_i,\cdots,M)$ by the property (iii). Proposition \ref{pd_preserved} implies that $\pd_{J(M_i,\cdots, M_r)} M_{i-1}\leq 1$ and thus 
    \begin{align*}
        \Hom_{J(M_i,\cdots, M_r)}(M_{i-1},\tau_{J(M_i,\cdots, M_r)}M_{i-1}) &\cong \Ext^1_{J(M_i,\cdots, M_r)}(M_{i-1},M_{i-1})\\
        &\cong \Ext_\Lambda^1(M_{i-1}, M_{i-1}) = 0.  
    \end{align*}
    This finishes the proof.
\end{proof}

The next result compares $\tau$-exceptional sequences and $R$-exceptional sequences in $\modd \Lambda$. 

\begin{proposition}\label{tau-exc=R-exc}
    $\tau$-exceptional sequences and $R$-exceptional sequences in $\modd \Lambda$ coincide. 
\end{proposition}

\begin{proof}
    Let $(Y_1,\cdots, Y_r)$ be an $R$-exceptional sequence in $\modd \Lambda$. Then, $\pd Y_i \leq 1$ for $1\leq i \leq r$ by \cite[Lemma 3.1(a)]{CB_RigidIntegralRep}, and thus $(Y_1,\cdots, Y_r)$ is a $\tau$-exceptional sequence in $\modd \Lambda$ by Lemma \ref{exceptional_with_pd<=1_is_tau-exc}.  

    Conversely, let $\X = (X_1,\cdots, X_r)$ be a $\tau$-exceptional sequence in $\modd \Lambda$. Then, $\X = (\Lambda\otimes_{kQ}M_1,\cdots, \Lambda\otimes_{kQ}M_r)$ for a unique exceptional sequence $(M_1,\cdots, M_r)$ in $\modd kQ$ by Theorem \ref{bijection_of_tau_exc_seq}. We show that $\X$ is an $R$-exceptional sequence in $\modd \Lambda$, that is $\Lambda\otimes_{kQ}M_i$ is an $R$-exceptional $RQ$-lattice for all $1\leq i \leq r$ and $\Hom_\Lambda(\Lambda\otimes_{kQ}M_i,\Lambda\otimes_{kQ}M_j) = 0 = \Ext_\Lambda^1(\Lambda\otimes_{kQ}M_i,\Lambda\otimes_{kQ}M_j)$ for $1\leq j <i \leq r$. Let $\overline{e_j}\in \Lambda$ be the trivial path at the vertex $j$. Then 
    $$ \overline{e_j}(\Lambda\otimes_{kQ}M_i) \cong (1\otimes e_j)(R\otimes kQ \otimes_{kQ} M_i) \cong (1\otimes e_j)(R\otimes M_i) \cong (R\otimes e_jM_i) \cong R^{\dim_k( e_jM_i)} $$
    where $e_j$ is the trivial path in $kQ$ at the vertex $j$. Since $\Lambda\otimes_{kQ}M_i = \bigoplus_{j=1}^n \overline{e_j}(\Lambda\otimes_{kQ}M_i)$ as an $R$-module, it is finitely generated and projective as an $R$-module. Moreover, using the fact that $M_i$ is an exceptional $kQ$-module, we obtain
    \begin{align*}
    \Ext_\Lambda^1(\Lambda\otimes_{kQ}M_i, \Lambda\otimes_{kQ}M_i) &\cong \Ext^1_{R\otimes kQ}(R\otimes M_i, R\otimes M_i)\\
    &\cong \Hom_R(R,R)\otimes \Ext_{kQ}^1(M_i,M_i) = 0,
    \end{align*}
    and 
    \begin{align*}
        \End_\Lambda(\Lambda\otimes_{kQ}M_i) &\cong \Hom_{R\otimes kQ}(R\otimes M_i, R\otimes M_i) \\
        &\cong \Hom_R(R,R)\otimes \Hom_{kQ}(M_i,M_i)\\
        &\cong R\otimes k \\
        &\cong R.
    \end{align*}
This shows that $\Lambda\otimes_{kQ}M_i$ is an $R$-exceptional $RQ$-lattice for all $1\leq i\leq r$. Since $(M_1,\cdots, M_r)$ is an exceptional sequence in $\modd kQ$, it follows that 
$$\Hom_\Lambda(\Lambda\otimes_{kQ}M_i,\Lambda\otimes_{kQ}M_j) \cong R\otimes \Hom_{kQ}(M_i, M_j) = 0,$$ and 
$$\Ext_\Lambda^1(\Lambda\otimes_{kQ}M_i,\Lambda\otimes_{kQ}M_j) = R\otimes \Ext_{kQ}^1(M_i,M_j) = 0$$
for all $1\leq j < i\leq r$. This proves that $\X$ is an $R$-exceptional sequence. The result follows. 
\end{proof}

Let $R$ be a commutative ring. Recall that a finitely generated $R$-module $P$ has \textit{constant rank} $n$ if $P_\mathfrak{p}$ is a free $R_\mathfrak{p}$-module of rank $n$ for each prime ideal $\mathfrak{p}$ of $R$, or equivalently for each maximal ideal. If $R$ has no idempotents other than $0$ and $1$, then every finitely generated projective $R$-module has constant rank (see \cite[Page 375]{CB_RigidIntegralRep}). In particular, this is the case when $R$ is a finite dimensional local commutative $k$-algebra. 

\begin{theorem}[{\cite[Thm. 4.2]{CB_RigidIntegralRep}}]\label{braid_group_on_R-exc}
    Let $R$ be a commutative ring and $Q$ be an acyclic quiver. Then, the braid group acts transitively on the set of complete $R$-exceptional sequences of pointwise constant rank over $RQ$. 
\end{theorem}

Let $\Lambda = R\otimes kQ$. Combining Proposition \ref{tau-exc=R-exc} and Theorem \ref{braid_group_on_R-exc}, we conclude that the braid group acts transitively on the set of complete $\tau$-exceptional sequences in $\modd \Lambda$. This provides a different approach to Corollary \ref{transitive_action_on_tau-exc_seq_over_Lambda}. In summary, we obtain the following result. 

\begin{theorem}\label{braid_action_on_R-exc=braid_action_on_tau-exc}
    Let $\Lambda = R\otimes kQ$. Then, the mutation of complete $R$-exceptional sequences in the sense of Crawley-Boevey coincides with the BHM mutation of complete $\tau$-exceptional sequences in $\modd \Lambda$. In particular, the braid actions on complete $R$-exceptional sequences and complete $\tau$-exceptional sequences coincide.
\end{theorem}

\begin{proof}
    Let $\X$ be an $R$-exceptional sequence. By Proposition \ref{tau-exc=R-exc} and Theorem \ref{bijection_of_tau_exc_seq}, $\X$ is of the form $(\Lambda\otimes_{kQ}M_1, \cdots, \Lambda\otimes_{kQ}M_n)$ for a unique exceptional sequence $(M_1,\cdots, M_n)$ in $\modd kQ$. Let $i\in \{1,\cdots, n\}$ and denote by $\psi_i(\X)$ the $i$-th left mutation of $\X$ in the sense of Crawley-Boevey. Then, 
    $$\psi_i(\X) = (\Lambda\otimes_{kQ}M_1,\cdots, \Lambda\otimes_{kQ}M_{i-1}, \widetilde{X}, \Lambda\otimes_{kQ}M_i,\cdots, \Lambda\otimes_{kQ}M_n),$$ 
    where $\widetilde{X}$ is described in \cite[Theorem 4.2]{CB_RigidIntegralRep}. On the other hand, by Theorem \ref{mutation_in_Lambda} we have 
    \begin{align*}
        \varphi_i(\X) &= \Lambda\otimes_{kQ}\sigma_i(M_1,\cdots, M_n)\\
        &= \Lambda\otimes_{kQ}(M_1,\cdots,M_{i-1},Y,M_i,\cdots, M_n)\\
        &= (\Lambda\otimes_{kQ}M_1, \cdots,\Lambda\otimes_{kQ}M_{i-1}, \Lambda\otimes_{kQ}Y,\Lambda\otimes_{kQ}M_i,\cdots, M_n),
    \end{align*}
     where $Y$ is the unique indecomposable exceptional $kQ$-module such that $\sigma_i(M_1,\cdots, M_n)$ is a complete exceptional sequence (see the beginning of Section \ref{secMutation}). It follows that $\psi_i(\X) = \varphi_i(\X)$ by Theorem \ref{uniqueness_property_of_tau_exc_seq}. In particular, the braid actions on complete $R$-exceptional and $\tau$-exceptional sequences coincide. 
\end{proof}

\section{An example}\label{secExample}

Let $t\geq 2 $. Define $\Lambda$ to be the tensor product of $R = k[x]/(x^t)$ with $kQ = k\Big(\begin{tikzcd}[ampersand replacement=\&,cramped]
	1 \& 2 \& 3
	\arrow["a", from=1-1, to=1-2]
	\arrow["b", from=1-2, to=1-3]
\end{tikzcd}\Big)$ over the field $k$. Then, $\Lambda \cong kQ_\Lambda/I$, where $Q_\Lambda$ is the quiver 

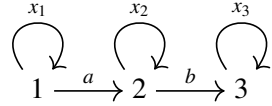
\[\begin{tikzcd}[ampersand replacement=\&,cramped]
	1 \& 2 \& 3
	\arrow["{x_1}", from=1-1, to=1-1, loop, in=55, out=125, distance=10mm]
	\arrow["a", from=1-1, to=1-2]
	\arrow["{x_2}", from=1-2, to=1-2, loop, in=55, out=125, distance=10mm]
	\arrow["b", from=1-2, to=1-3]
	\arrow["{x_3}", from=1-3, to=1-3, loop, in=55, out=125, distance=10mm]
\end{tikzcd}\]

and $I$ is the ideal generated by $(x_i^t, ax_1-x_2a, bx_2-x_3b)$, for $i = 1,2,3$. In particular, $\Lambda$ is tame for $t = 4$ and wild for $t\geq 5$ (see \cite[Prop. 4.1]{wang2024representation}).

Notice that every indecomposable $kQ$-module is $\tau$-rigid. Moreover, it is well-known that $P\otimes Q$ is a projective (resp. injective) $\Lambda$-module if and only if $P$ and $Q$ are projective (resp. injective) $kQ$-modules. Hence, since $\Lambda\otimes_{kQ}X  = R\otimes (kQ\otimes_{kQ} X) \cong R\otimes X$ for every $kQ$-module $X$ and, $R$ is projective-injective, using Proposition \ref{tau-rigid-Lambda1:1tau-rigid-kQ} the indecomposable $\tau$-rigid $\Lambda$-modules are given by $P_i, I_i$, where $P_i$ (resp. $I_i$) denotes the indecomposable projective (resp. injective) $\Lambda$-module at the vertex $i$, and $M = \Lambda\otimes_{kQ}S_2$, where $S_2$ is the simple $kQ$-module at vertex 2.

Now, using Theorem \ref{bijection_of_tau_exc_seq} and Theorem \ref{mutation_in_Lambda}, the mutation graph of complete $\tau$-exceptional sequences in $\modd \Lambda$ can be computed by applying the induction functor to the mutation graph of complete ($\tau$-)exceptional sequences in $\modd kQ$ (see Figure \ref{mutation_graph}; the action of $\varphi_1$ is indicated by solid arrows, while dashed arrows indicate the action of $\varphi_2$). In particular, the braid relations hold, that is $\varphi_1\circ\varphi_2\circ\varphi_1\widetilde{\M} = \varphi_2\circ\varphi_1\circ\varphi_2\widetilde{\M}$ for every complete $\tau$-exceptional sequence $\widetilde{\M}$ in $\modd \Lambda$ (see Corollary \ref{transitive_action_on_tau-exc_seq_over_Lambda}).
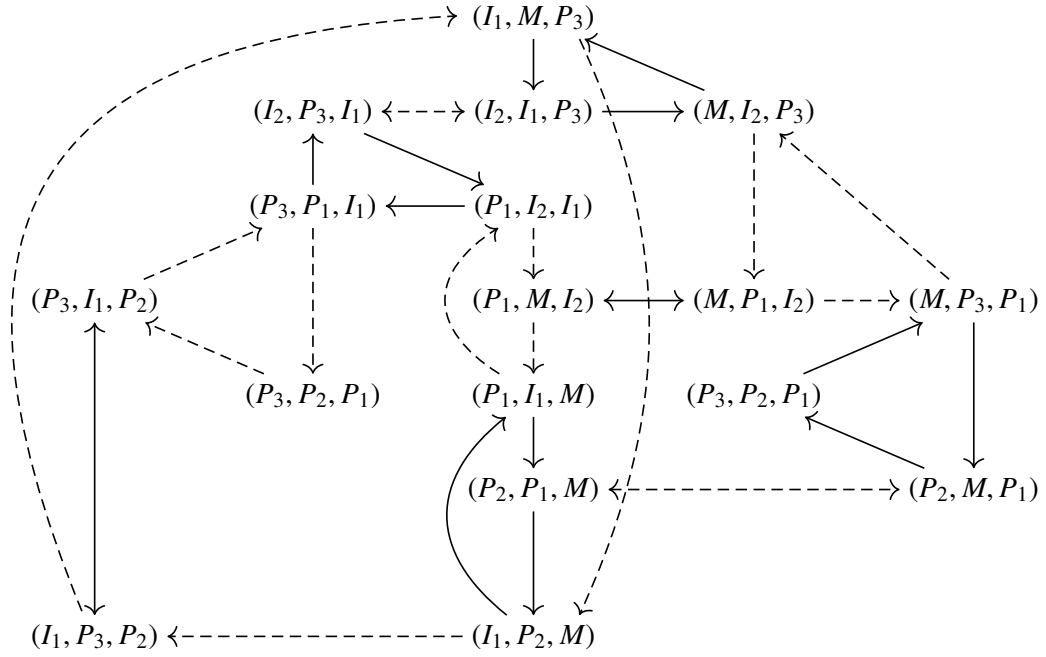
\begin{figure}[h]
    \centering
\[\begin{tikzcd}[ampersand replacement=\&,cramped]
	\&\& {(I_1, M, P_3)} \\
	\& {(I_2, P_3, I_1)} \& {(I_2, I_1, P_3)} \& {(M,I_2,P_3)} \\
	\& {(P_3,P_1,I_1)} \& {(P_1,I_2,I_1)} \\
	{(P_3,I_1,P_2)} \&\& {(P_1,M,I_2)} \& {(M,P_1,I_2)} \& {(M,P_3,P_1)} \\
	\& {(P_3,P_2,P_1)} \& {(P_1,I_1,M)} \& {(P_3,P_2,P_1)} \\
	\&\& {(P_2,P_1,M)} \&\& {(P_2,M,P_1)} \\
	\\
	{(I_1, P_3, P_2)} \&\& {(I_1,P_2,M)}
	\arrow[from=1-3, to=2-3]
	\arrow[dashed, shift left=5, curve={height=-35pt}, from=1-3, to=8-3]
	\arrow[from=2-2, to=3-3]
	\arrow[dashed, tail reversed, from=2-3, to=2-2]
	\arrow[from=2-3, to=2-4]
	\arrow[from=2-4, to=1-3]
	\arrow[dashed, from=2-4, to=4-4]
	\arrow[from=3-2, to=2-2]
	\arrow[dashed, from=3-2, to=5-2]
	\arrow[from=3-3, to=3-2]
	\arrow[dashed, from=3-3, to=4-3]
	\arrow[dashed, from=4-1, to=3-2]
	\arrow[tail reversed, from=4-3, to=4-4]
	\arrow[dashed, from=4-3, to=5-3]
	\arrow[dashed, from=4-4, to=4-5]
	\arrow[dashed, from=4-5, to=2-4]
	\arrow[from=4-5, to=6-5]
	\arrow[dashed, from=5-2, to=4-1]
	\arrow[curve={height=-40pt}, dashed, from=5-3, to=3-3]
	\arrow[from=5-3, to=6-3]
	\arrow[from=5-4, to=4-5]
	\arrow[dashed, tail reversed, from=6-3, to=6-5]
	\arrow[from=6-3, to=8-3]
	\arrow[from=6-5, to=5-4]
	\arrow[dashed, shift left = 1, curve={height=-130pt}, from=8-1, to=1-3]
	\arrow[ tail reversed, from=8-1, to=4-1]
	\arrow[curve={height=-40pt}, from=8-3, to=5-3]
	\arrow[dashed, from=8-3, to=8-1]
\end{tikzcd}\]
    \caption{Mutation graph of complete $\tau$-exceptional sequences in $\modd \Lambda$.}
    \label{mutation_graph}
\end{figure}

\vspace{10pt}

\paragraph{\textbf{Acknowledgment:}} I express my gratitude to Bethany R. Marsh for her supervision, for our countless inspirational discussions, and her meticulous proofreading. I am also grateful for financial support from the University of Leeds and the EPSRC Programme Grant EP/W007509/1.
\maketitle

\nocite{*}
\bibliographystyle{alpha}
\bibliography{references}

\newcommand{\etalchar}[1]{$^{#1}$}
\begin{thebibliography}{BMR{\etalchar{+}}06}

\bibitem[AIR14]{tau-tiling-theory}
Takahide Adachi, Osamu Iyama, and Idun Reiten.
\newblock {$\tau$}-tilting theory.
\newblock {\em Compos. Math.}, 150(3):415--452, 2014.

\bibitem[AS81]{AlmostSplitSeqInSubcategories}
M.~Auslander and Sverre~O. Smal\o.
\newblock Almost split sequences in subcategories.
\newblock {\em J. Algebra}, 69(2):426--454, 1981.

\bibitem[ASS06]{ASS}
Ibrahim Assem, Daniel Simson, and Andrzej Skowro\'{n}ski.
\newblock {\em Elements of the representation theory of associative algebras. {V}ol. 1}, volume~65 of {\em London Mathematical Society Student Texts}.
\newblock Cambridge University Press, Cambridge, 2006.
\newblock Techniques of representation theory.

\bibitem[BH22]{barnard2022exceptional}
Emily Barnard and Eric~J Hanson.
\newblock Exceptional sequences in semidistributive lattices and the poset topology of wide subcategories.
\newblock {\em arXiv:2209.11734. To appear in J. Pure Appl. Algebra.}, 2022.

\bibitem[BH23]{tau-perpendicular_wide_subategories}
Aslak~Bakke Buan and Eric~J. Hanson.
\newblock {$\tau$}-perpendicular wide subcategories.
\newblock {\em Nagoya Math. J.}, 252:959--984, 2023.

\bibitem[BHM24]{BHM_mutation}
Aslak~B Buan, Eric~J Hanson, and Bethany~R. Marsh.
\newblock Mutation of ${\tau}$-exceptional pairs and sequences.
\newblock {\em arXiv:2402.10301}, 2024.

\bibitem[BKT25]{BKH_mutating_ordered-tau-rigid}
Aslak~B Buan, Maximilian Kaipel, and H{\aa}vard~U Terland.
\newblock Mutating ordered $\tau$-rigid modules with applications to {Nakayama} algebras.
\newblock {\em arXiv preprint arXiv:2501.13694}, 2025.

\bibitem[BM21a]{a_category_of_wide_subcategories}
Aslak~Bakke Buan and Bethany~R. Marsh.
\newblock A category of wide subcategories.
\newblock {\em Int. Math. Res. Not. IMRN}, (13):10278--10338, 2021.

\bibitem[BM21b]{tauExcSeq_BM}
Aslak~Bakke Buan and Bethany~Rose Marsh.
\newblock $\tau$-exceptional sequences.
\newblock {\em J. Algebra}, 585:36--68, 2021.

\bibitem[BM23]{Mutating-signed-tau-exceptional-sequences}
Aslak~Bakke Buan and Bethany~Rose Marsh.
\newblock Mutating signed {$\tau$}-exceptional sequences.
\newblock {\em Glasg. Math. J.}, 65(3):716--729, 2023.

\bibitem[BMR{\etalchar{+}}06]{Tilting_Theory_and_Cluster_Combinatorics}
Aslak~Bakke Buan, Bethany Marsh, Markus Reineke, Idun Reiten, and Gordana Todorov.
\newblock Tilting theory and cluster combinatorics.
\newblock {\em Adv. Math.}, 204(2):572--618, 2006.

\bibitem[Bon89]{Bondal}
A.~I. Bondal.
\newblock Representations of associative algebras and coherent sheaves.
\newblock {\em Izv. Akad. Nauk SSSR Ser. Mat.}, 53(1):25--44, 1989.

\bibitem[B{\o}r21]{borve21_two-term}
Erlend~D B{\o}rve.
\newblock Two-term silting and $\tau$-cluster morphism categories.
\newblock {\em arXiv preprint arXiv:2110.03472}, 2021.

\bibitem[BRT11]{ThreeKindsOfMutations}
Aslak~Bakke Buan, Idun Reiten, and Hugh Thomas.
\newblock Three kinds of mutation.
\newblock {\em J. Algebra}, 339:97--113, 2011.

\bibitem[BST19]{Wall_Chamber_Structure}
Thomas Br\"{u}stle, David Smith, and Hipolito Treffinger.
\newblock Wall and chamber structure for finite-dimensional algebras.
\newblock {\em Adv. Math.}, 354:106746, 31, 2019.

\bibitem[CB92]{Boevey-ExcSeq}
William Crawley-Boevey.
\newblock Exceptional sequences of representations of quivers.
\newblock In {\em Proceedings of the {S}ixth {I}nternational {C}onference on {R}epresentations of {A}lgebras ({O}ttawa, {ON}, 1992)}, volume~14 of {\em Carleton-Ottawa Math. Lecture Note Ser.}, page~7. Carleton Univ., Ottawa, ON, 1992.

\bibitem[CB24]{CB_RigidIntegralRep}
William Crawley-Boevey.
\newblock Rigid integral representations of quivers over arbitrary commutative rings.
\newblock {\em Annals of Representation Theory}, 1(3):375--384, 2024.

\bibitem[CE99]{CartanEilenberg}
Henri Cartan and Samuel Eilenberg.
\newblock {\em Homological algebra}.
\newblock Princeton Landmarks in Mathematics. Princeton University Press, Princeton, NJ, 1999.
\newblock With an appendix by David A. Buchsbaum, Reprint of the 1956 original.

\bibitem[DIJ19]{tauTiltingFiniteAlgebrasAnd_g-vectors}
Laurent Demonet, Osamu Iyama, and Gustavo Jasso.
\newblock {$\tau$}-tilting finite algebras, bricks, and {$g$}-vectors.
\newblock {\em Int. Math. Res. Not. IMRN}, (3):852--892, 2019.

\bibitem[DIR{\etalchar{+}}23]{Beyond_tau-tilting_theory}
Laurent Demonet, Osamu Iyama, Nathan Reading, Idun Reiten, and Hugh Thomas.
\newblock Lattice theory of torsion classes: beyond {$\tau$}-tilting theory.
\newblock {\em Trans. Amer. Math. Soc. Ser. B}, 10:542--612, 2023.

\bibitem[EJR18]{ReductionTheoremForTauRigidModules}
Florian Eisele, Geoffrey Janssens, and Theo Raedschelders.
\newblock A reduction theorem for {$\tau$}-rigid modules.
\newblock {\em Math. Z.}, 290(3-4):1377--1413, 2018.

\bibitem[Eno22]{Rigid_ICE-subcategories}
Haruhisa Enomoto.
\newblock Rigid modules and {ICE}-closed subcategories in quiver representations.
\newblock {\em J. Algebra}, 594:364--388, 2022.

\bibitem[ES22]{Enomoto-Sakai}
Haruhisa Enomoto and Arashi Sakai.
\newblock I{CE}-closed subcategories and wide {$\tau$}-tilting modules.
\newblock {\em Math. Z.}, 300(1):541--577, 2022.

\bibitem[GL91]{Geigle_Lenzing_Perp_Subcat}
Werner Geigle and Helmut Lenzing.
\newblock Perpendicular categories with applications to representations and sheaves.
\newblock {\em J. Algebra}, 144(2):273--343, 1991.

\bibitem[GLS17]{Quivers_with_relations_for_symmetrizable_Cartan_matrices_I}
Christof Geiss, Bernard Leclerc, and Jan Schr\"{o}er.
\newblock Quivers with relations for symmetrizable {C}artan matrices {I}: {F}oundations.
\newblock {\em Invent. Math.}, 209(1):61--158, 2017.

\bibitem[Gor88]{Gorodentsev}
A.~L. Gorodentsev.
\newblock Exceptional bundles on surfaces with a moving anticanonical class.
\newblock {\em Izv. Akad. Nauk SSSR Ser. Mat.}, 52(4):740--757, 895, 1988.

\bibitem[GR87]{GR}
A.~L. Gorodentsev and A.~N. Rudakov.
\newblock Exceptional vector bundles on projective spaces.
\newblock {\em Duke Math. J.}, 54(1):115--130, 1987.

\bibitem[HI21a]{Pairwise_compatibility_for_2-simple_minded_collections}
Eric~J. Hanson and Kiyoshi Igusa.
\newblock Pairwise compatibility for 2-simple minded collections.
\newblock {\em J. Pure Appl. Algebra}, 225(6):Paper No. 106598, 31, 2021.

\bibitem[HI21b]{Hanson-Igusa_tau-cluster_morphism_categories_and_picture_groups}
Eric~J. Hanson and Kiyoshi Igusa.
\newblock {$\tau$}-cluster morphism categories and picture groups.
\newblock {\em Comm. Algebra}, 49(10):4376--4415, 2021.

\bibitem[HS97]{A_Course_In_Homological_Algebra}
P.~J. Hilton and U.~Stammbach.
\newblock {\em A course in homological algebra}, volume~4 of {\em Graduate Texts in Mathematics}.
\newblock Springer-Verlag, New York, second edition, 1997.

\bibitem[HT24]{uniqueness_of_tau_exc_seq}
Eric~J. Hanson and Hugh Thomas.
\newblock A uniqueness property of {$\tau$}-exceptional sequences.
\newblock {\em Algebr. Represent. Theory}, 27(1):461--468, 2024.

\bibitem[IT09]{Noncrossing_partitions_and_representations_of_quivers}
Colin Ingalls and Hugh Thomas.
\newblock Noncrossing partitions and representations of quivers.
\newblock {\em Compos. Math.}, 145(6):1533--1562, 2009.

\bibitem[IT17]{signed_exc_seq}
K.~Igusa and G.~Todorov.
\newblock Signed exceptional sequences and the cluster morphism category.
\newblock preprint arXiv:1706.02041v1 [math.RT], 2017.

\bibitem[ITW16]{igusa2016picture}
Kiyoshi Igusa, Gordana Todorov, and Jerzy Weyman.
\newblock Picture groups of finite type and cohomology in type ${A_n}$.
\newblock {\em arXiv:1609.02636}, 2016.

\bibitem[Jas15]{Jasso_Reduction}
Gustavo Jasso.
\newblock Reduction of {$\tau$}-tilting modules and torsion pairs.
\newblock {\em Int. Math. Res. Not. IMRN}, (16):7190--7237, 2015.

\bibitem[Kai25]{Kaipel_the_category_of_a_partitioned_fan}
Maximilian Kaipel.
\newblock The category of a partitioned fan.
\newblock {\em J. Lond. Math. Soc. (2)}, 111(2):Paper No. e70071, 2025.

\bibitem[Les94]{On_the_representation_type_of_tensor_product_algebras}
Zbigniew Leszczy\'nski.
\newblock On the representation type of tensor product algebras.
\newblock {\em Fund. Math.}, 144(2):143--161, 1994.

\bibitem[Lu22]{TensorOfHigherAPR-Tilting}
Xiaojian Lu.
\newblock Tensor products of higher {APR} tilting modules.
\newblock {\em arXiv:2211.04962}, 2022.

\bibitem[MHT20]{Stratifying_Sistems_Hipolito}
Octavio Mendoza~Hern\'{a}ndez and Hipolito Treffinger.
\newblock Stratifying systems through {$\tau$}-tilting theory.
\newblock {\em Doc. Math.}, 25:701--720, 2020.

\bibitem[M{\v{S}}17]{Torsion-classes-wide-subcategories-and-localisations}
Frederik Marks and Jan {\v{S}}{\v{t}}ov{\'\i}{\v{c}}ek.
\newblock Torsion classes, wide subcategories and localisations.
\newblock {\em Bull. Lond. Math. Soc.}, 49(3):405--416, 2017.

\bibitem[Non25]{nonis2025tau}
Iacopo Nonis.
\newblock $\tau$-exceptional sequences for representations of quivers over local algebras.
\newblock {\em arXiv preprint arXiv:2502.15417}, 2025.

\bibitem[Rin94]{braid_group_action_on_exc_seq_Ringel}
Claus~Michael Ringel.
\newblock The braid group action on the set of exceptional sequences of a hereditary {A}rtin algebra.
\newblock In {\em Abelian group theory and related topics ({O}berwolfach, 1993)}, volume 171 of {\em Contemp. Math.}, pages 339--352. Amer. Math. Soc., Providence, RI, 1994.

\bibitem[RZ17]{RepOfQuiversOverTheDualNumbers}
Claus~Michael Ringel and Pu~Zhang.
\newblock Representations of quivers over the algebra of dual numbers.
\newblock {\em J. Algebra}, 475:327--360, 2017.

\bibitem[STTW23]{a_geometric_perspective_on_the_tau_cluster_morphism_category}
Sibylle Schroll, Aran Tattar, Hipolito Treffinger, and Nicholas~J Williams.
\newblock A geometric perspective on the $\tau$-cluster morphism category.
\newblock {\em arXiv:2302.12217}, 2023.

\bibitem[Wan24]{wang2024representation}
Qi~Wang.
\newblock Representation-finite tensor product algebras.
\newblock {\em arXiv preprint arXiv:2407.11363}, 2024.

\end{thebibliography}

\end{document}